\documentclass[11pt,a4paper, reqno]{article}
\usepackage[utf8]{inputenc}
\usepackage{amsfonts}
\usepackage{amssymb}
\usepackage{enumerate}
\usepackage{graphicx}
\usepackage{mathtools}
\usepackage{amsmath, amsthm}
\usepackage{tikz}
\usepackage{color}
\usepackage[affil-it]{authblk}
\usepackage[T1]{fontenc}
\usepackage[numbers,sort]{natbib}
\usepackage{bbm}
\usepackage[a4paper, total={6.5in, 9.5in}]{geometry}
\usepackage{amsmath}
\usepackage{amsfonts}
\usepackage{amssymb}
\usepackage{undertilde}
\usepackage{bbm}
\usepackage{mathrsfs}
\usepackage[utf8]{inputenc}
\usepackage[english]{babel}
\usepackage{hyperref}
\usepackage{relsize}
\usepackage{accents}
\usepackage{chngcntr}

\usepackage{palatino}

\makeatletter
\newsavebox\myboxA
\newsavebox\myboxB
\newlength\mylenA

\newcommand\blfootnote[1]{%
  \begingroup
  \renewcommand\thefootnote{}\footnote{#1}%
  \addtocounter{footnote}{-1}%
  \endgroup
}   

\newcommand*\xoverline[2][0.75]{%
    \sbox{\myboxA}{$\m@th#2$}%
    \setbox\myboxB\null
    \ht\myboxB=\ht\myboxA%
    \dp\myboxB=\dp\myboxA%
    \wd\myboxB=#1\wd\myboxA
    \sbox\myboxB{$\m@th\overline{\copy\myboxB}$}
    \setlength\mylenA{\the\wd\myboxA}
    \addtolength\mylenA{-\the\wd\myboxB}%
    \ifdim\wd\myboxB<\wd\myboxA%
       \rlap{\hskip 0.5\mylenA\usebox\myboxB}{\usebox\myboxA}%
    \else
        \hskip -0.5\mylenA\rlap{\usebox\myboxA}{\hskip 0.5\mylenA\usebox\myboxB}%
    \fi}
\makeatother

\long\def\/*#1*/{}

\usepackage{hyperref}
\hypersetup{
  colorlinks,
  linkcolor=blue,
  citecolor=black,
  filecolor=black,
  urlcolor=black,
  pdftitle={},
  pdfauthor={S. Dhara, R. van der Hofstad, J. van Leeuwaarden, S. Sen},
  pdfcreator={S. Dhara},
  pdfsubject={},
  pdfkeywords={}
}

\numberwithin{equation}{section}
\usepackage{cleveref}

\newcommand*{\Mtilde}[1]{\skew{5}{\tilde}{#1}}
\newcommand{\prob}[1]{\ensuremath{\mathbbm{P}\left(#1\right)}}
\newcommand{\expt}[1]{\ensuremath{\mathbbm{E}\left[#1\right]}}

\newcommand{\ind}[1]{\ensuremath{\mathbbm{1}_{\left\{#1\right\}}}}
\newcommand{\pto}{\ensuremath{\xrightarrow{\mathbbm{P}}}}
\newcommand{\dto}{\ensuremath{\xrightarrow{\mathcal{L}}}}
\newcommand{\surp}[1]{\ensuremath{\mathrm{SP}(#1)}}

\newcommand{\PR}{\ensuremath{\mathbbm{P}}}
\newcommand{\E}{\ensuremath{\mathbbm{E}}}
\newcommand{\R}{\ensuremath{\mathbbm{R}}}

\newcommand{\e}{\ensuremath{\mathrm{e}}}
\newcommand{\1}{\ensuremath{\mathbbm{1}}}
\newcommand{\OP}{\ensuremath{O_{\sss\PR}}}
\newcommand{\oP}{\ensuremath{o_{\sss\PR}}}

\newcommand{\CM}{\ensuremath{\mathrm{CM}_n(\boldsymbol{d})}}

\newcommand{\bld}[1]{\boldsymbol{#1}}
\newcommand{\shortarrow}{\ensuremath{{\sss \downarrow}}}
\newcommand{\rE}{\mathrm{E}}
\newcommand{\rD}{\mathrm{D}}
\newcommand{\dif}{\mathrm{d}}
\newcommand{\ord}{\mathrm{ord}}

\newtheorem{theorem}{Theorem}
\newtheorem{algo}{Algorithm}
\newtheorem{lemma}[theorem]{Lemma}
\newtheorem{proposition}[theorem]{Proposition}

\newtheorem{assumption}{Assumption}
\newtheorem{remark}{Remark}

\begin{document}

\def\sss{\scriptscriptstyle}

\title{Critical window for the configuration model: \\finite third moment degrees}
\author{Souvik Dhara*}
\author{Remco van der Hofstad*}
\author{Johan S.H. van Leeuwaarden*}
\author{Sanchayan Sen$^\dagger$}
\affil{*Department of Mathematics and Computer Science,

Eindhoven University of Technology\\

$^\dagger$  Department of Mathematics and Statistics, McGill University }
\renewcommand\Authands{, }
\date{\today}
\maketitle
\blfootnote{\emph{Correspondence to:} \href{mailto:s.dhara@tue.nl}{S. Dhara}.}
\blfootnote{\emph{Email adresses:}  \href{mailto:s.dhara@tue.nl}{s.dhara@tue.nl}, \href{mailto:r.w.v.d.hofstad@tue.nl}{r.w.v.d.hofstad@tue.nl}, \href{mailto:j.s.h.v.leeuwaarden@tue.nl}{j.s.h.v.leeuwaarden@tue.nl}, \href{mailto:sanchayan.sen1@gmail.com}{sanchayan.sen1@gmail.com}.}
\blfootnote{\emph{Keywords and phrase}. Critical configuration model, finite third moment degree, Brownian excursions with parabolic drift, scaling window, multiplicative coalescent, universality.}
\begin{abstract}
   We investigate the component sizes of the critical configuration model, as well as the related problem of critical percolation  on a supercritical configuration model.
 We show that, at criticality, the finite third moment assumption on the asymptotic degree distribution is enough to guarantee that the sizes of the largest connected components are of the order $n^{2/3}$ and the re-scaled component sizes (ordered in a decreasing manner) converge to the ordered excursion lengths of an inhomogeneous Brownian Motion with a parabolic drift. We use percolation to study the evolution of these component sizes while passing through the critical window and  show that the vector of percolation cluster-sizes, considered as a process in the critical window, converge to the multiplicative coalescent process in the sense of finite dimensional distributions. This behavior was first observed for Erd\H{o}s-R\'enyi random graphs by Aldous (1997) and our results provide support for the empirical evidences that the nature of the phase transition for a wide array of random-graph models are universal in nature. Further, we show that the re-scaled component sizes and surplus edges converge jointly under a strong topology, at each fixed location of the scaling window.
\end{abstract}


\section{Introduction}
\label{secintro}
   Random graphs are the main vehicles to study complex networks that go through  a radical change in their connectivity, often called the \emph{phase-transition}.  
   A large body of literature aims at understanding the properties of random graphs that experience this phase-transition in the sizes of the large connected components for various models. The behavior is well understood for the Erd\H{o}s-R\'enyi random graphs, thanks to a plethora of results \cite{A97,JLR00,NP10a,HJL10}. 
However, these graphs are often inadequate for modeling real-world networks \cite{FFF99,SHL16,CSN09,SHL16c} since the real-world network data often show a power-law behavior of the asymptotic degrees  whereas the degree distribution of the Erd\H{o}s-R\'enyi random graphs has exponentially decaying tails. 
Therefore, many alternative models have been proposed to capture this power-law tail behavior. 
An interesting fact, however, is that the behavior, in most of these models, is quite universal in the sense that there is a critical value where the graphs experience a phase-transition and the nature of this phase-transition is insensitive to the microscopic descriptions of the model \cite{DLV14,BHL10,NP10a,R12,Jo10,AP00,SHL16b}. \par

In this work, we focus on the \emph{configuration model}, the canonical model for generating a random multi-graph with a prescribed degree sequence. 
This model was introduced by Bollob\'as~\cite{B80}  to choose a uniform simple $d$-regular graph on $n$ vertices, when $dn$ is even. The idea was later generalized for general degree sequences $\boldsymbol{d}$ by Molloy and Reed~\cite{MR95} and others. 
We denote by $\mathrm{CM}_n(\boldsymbol{d})$ the  multi-graph generated by the configuration model on the vertex set $[n]=\{1,2\dots,n\}$ with the degree sequence $\boldsymbol{d}.$ 
The configuration model, conditioned on simplicity, yields a {\em uniform} simple graph with the same degree sequence.  
Various features related to the emergence of the giant component phenomenon for this model have been studied recently \cite{F07, J09, MR95, JL09, Jo10, R12}. We give a brief overview of the relevant literature in Section \ref{sec_literature_overview}.
Our aim is to obtain precise asymptotics  for the component sizes of  $\mathrm{CM}_n(\boldsymbol{d})$ in the \emph{critical window} of phase transition under the optimal assumptions on the degree sequence involving a finite third-moment condition. 
The re-scaled vector of component sizes (ordered in a decreasing manner) is shown to converge to the ordered excursion lengths of certain reflected inhomogeneous Brownian motions with a parabolic drift. 
This shows that the component sizes of $\mathrm{CM}_n(\boldsymbol{d})$ in the critical regime, for a large collection of possible $\boldsymbol{d}$, lies in the same universality class as the Erd\H{o}s-R\'enyi random graph~\cite{A97} and the inhomogeneous random graph~\cite{BHL10}.  
We use percolation on a super-critical configuration model to show  the joint convergence of the scaled vectors of component sizes at multiple locations of the percolation scaling window. We also obtain the asymptotic distribution of the number of surplus edges in each component and show that the sequence of vectors consisting of the re-scaled component sizes and surplus converges to a suitable limit under a strong topology as discussed in~\cite{BBW12}. These results give very strong evidence in favor of the structural similarity of the component sizes of $\mathrm{CM}_n(\boldsymbol{d})$ and Erd\H{o}s-R\'enyi random graphs at criticality. 
\subsubsection*{Our contribution} The main contribution of this paper is that we derive the strongest results in the literature under the finite third-moment assumption on the degrees. 
This finite third-moment assumption is also necessary for Erd\H{o}s-R\'enyi type scaling limits, since, amongst other reasons, the third moment appears in the scaling limit.
In a recent work \cite{DHLS16}, we consider the infinite third-moment case with power-law degrees and show that the scaling limit of the cluster sizes is quite different. 
Also, we prove the joint convergence of the component sizes and the surplus edges under a strong topology, which improves the previous known results involving the surplus edges~\cite{R12}. 
We also study percolation on the configuration model to gain insight about the evolution of the configuration model over the critical scaling window. 
This is achieved by studying a dynamic process that generates the percolated graphs with different values of the percolation parameter, a problem that is interesting in its own right. 

Before stating our main results, we need to introduce some notation and concepts. \nocite{BHL12}

\section{Definitions and notation} \label{sub_sec_notation} We will use the standard notation  $\xrightarrow{\mathbbm{P}}$, $\xrightarrow{\mathcal{L}}$ to denote convergence in probability and in distribution or law, respectively. 
We often use the Bachmann Landau notation $O(\cdot)$, $o(\cdot)$ for large $n$ asymptotics of real numbers.
The topology needed for the distributional convergence will always be specified unless it is  clear from the context.  A sequence of events $(\mathcal{E}_n)_{n\geq 1}$ is said to occur with high probability (whp) with respect to probability measures $(\mathbbm{P}_n)_{n\geq 1}$  if $\mathbbm{P}_n\big( \mathcal{E}_n \big) \to 1$. Denote $f_n = O_{\sss \mathbbm{P}}(g_n)$ if $ ( |f_n|/|g_n| )_{n \geq 1} $ is tight; $f_n =o_{\sss \mathbbm{P}}(g_n)$ if $(|f_n|/|g_n|)_{n\geq 1}$ converges in probability to zero; $f_n =\Theta_{\sss \mathbbm{P}}(g_n)$ if $f_n=O_{\sss \mathbbm{P}}(g_n)$ and $g_n=O_{\sss \mathbbm{P}}(f_n)$. 
For a triangular array of random variables $(f_{k,n})_{k,n\geq 1}$, we write phrases like $f_{k,n} = \OP(n^{\alpha})$ (respectively  $\oP(n^{\alpha})$), uniformly over $k\leq n^{\beta}$ to mean that $\sup_{k\leq n^{\alpha}}|f_{k,n}| = \OP(n^{\alpha})$ (respectively $\oP(n^{\alpha})$).
We also write $f_n=O_{\sss E}(a_n)$ (respectively $f_n=o_{\sss E}(a_n)$) to denote that $\sup_{n\geq 1}\expt{a_n^{-1}f_n}<\infty$ (respectively $\lim_{n\to\infty}\expt{a_n^{-1}f_n}=0$). Denote by 
\begin{equation}\ell^2_{\shortarrow}:= \big\{ \mathbf{x}= (x_1, x_2, x_3, ...): x_1 \geq x_2 \geq x_3 \geq ... \text{ and } \sum_{i=1}^{\infty} x_{i}^2 < \infty \big\},
\end{equation}the subspace of non-negative, non-increasing  sequences of real numbers with square norm metric $d(\mathbf{x}, \mathbf{y})=( \sum_{i=1}^{\infty} (x_i-y_i)^2 )^{1/2}$ and let $(\ell^2_{\shortarrow})^k$ denote the $k$-fold product space of $\ell^2_{\shortarrow}$. 
With $\ell^2_{\shortarrow} \times \mathbbm{N}^{\infty}$, we denote the product topology of $\ell^2_{\shortarrow}$ and $\mathbbm{N}^{\infty}$, where $\mathbbm{N}^{\infty}$ denotes the collection of sequences on~$\mathbbm{N}$, endowed with the product topology. Define also
\begin{equation}
\mathbb{U}_{\shortarrow}:= \big\{ ((x_i,y_i))_{i=1}^{\infty}\in  \ell^2_{\shortarrow} \times \mathbbm{N}^{\infty}: \sum_{i=1}^{\infty} x_iy_i < \infty \text{ and } y_i=0 \text{ whenever } x_i=0, \ \forall i   \big\}
\end{equation} with the metric \begin{equation} \label{defn_U_metric}d_{\mathbb{U}}((\mathbf{x}_1, \mathbf{y}_1), (\mathbf{x}_2, \mathbf{y}_2)):= \bigg( \sum_{i=1}^{\infty} (x_{1i}-x_{2i})^2 \bigg)^{1/2}+ \sum_{i=1}^{\infty} \big| x_{1i} y_{1i} - x_{2i}y_{2i}\big|.
\end{equation} Further, we introduce $\mathbb{U}^0_{\shortarrow} \subset \mathbb{U}_{\shortarrow}$ as \begin{equation}\mathbb{U}^0_{\shortarrow}:= \big\{((x_i,y_i))_{i=1}^{\infty}\in\mathbb{U}_{\shortarrow} : \text{ if } x_k = x_m, k \leq m,\text{ then }y_k \geq y_m\big\}.
\end{equation}
 We usually use the boldface notation $\mathbf{X}$ for a time-dependent stochastic process $( X(s))_{s \geq 0}$, unless stated otherwise, $\mathbb{C}[0,t]$  denotes the set of all continuous functions  from $[0,t]$ to $\mathbbm{R}$ equipped with the topology induced by sup-norm $||\cdot||_{t}$. 
Similarly, $\mathbbm{D}[0,t]$ (resp.~$\mathbbm{D}[0,\infty)$) denotes the set of all c\`adl\`ag functions from $[0,t]$ (resp.~$[0,\infty)$) to $\mathbbm{R}$ equipped with the Skorohod $J_1$ topology. 
$\mathbf{B}^{\lambda}_{\mu,\eta}$ denotes an inhomogeneous Brownian motion with a parabolic drift, given by
\begin{equation}\label{def:inhomogen:BM}
B^{\lambda}_{\mu,\eta}(s)=\frac{\sqrt{\eta}}{\mu} B(s) +\lambda s-\frac{\eta s^2}{2\mu^{3}}
\end{equation}
 where $\mathbf{B}= ( B(s) )_{s \geq 0}$ is a standard Brownian motion, and $\mu>0$, $\eta>0$ and $\lambda\in \R$ are constants. Define the reflected version of $\mathbf{B}^{\lambda}_{\mu,\eta}$ as
\begin{equation} \label{defn::reflected-BM}
W^{\lambda}(s) = B^{\lambda}_{\mu,\eta}(s) - \min_{0 \leq t \leq s} B^{\lambda}_{\mu,\eta}(t).
\end{equation}
For a function $f\in \mathbb{C}[0,\infty)$, an interval $\gamma=(l,r)$ is called an \emph{excursion above past minima} or simply an \emph{excursion} of $f$ if $f(l)=f(r)=\min_{u\leq r}f(u)$ and $f(x)>f(r)$ for all $l<x<r$.   $|\gamma|=r(\gamma)-l(\gamma)$ will denote the length of the excursion $\gamma$.
\par Also, define the counting process of marks $\mathbf{N}^\lambda= ( N^\lambda(s) )_{s \geq 0}$ to be a unit-jump process with intensity $\beta W^\lambda(s)$ at time $s$ conditional on $( W^\lambda(u) )_{u \leq s}$ so that
\begin{equation} \label{defn::counting-process}
N^\lambda(s) - \int\limits_{0}^{s} \beta W^\lambda(u)du
\end{equation} is a martingale (see \citep{A97}). 
For an excursion $\gamma$, let $N(\gamma)$ denote the number of marks in the interval $[l(\gamma),r(\gamma)]$.
\begin{remark} \label{defn_U_0_process}\normalfont By \cite[Lemma 25]{A97}, the excursion lengths  of $\mathbf{B}^{\lambda}_{\mu,\eta}$ can be rearranged in decreasing order of length and the ordered excursion lengths can be considered as a vector in $\ell^2_{\shortarrow}$, almost surely. Let $\boldsymbol{\gamma}^\lambda = ( | \gamma^\lambda_{j}|)_{j\geq1}$ be the ordered excursion lengths of $\mathbf{B}^{\lambda}_{\mu,\eta}$. Then,  $( | \gamma_j^\lambda | , N(\gamma_j^\lambda))_{ j \geq 1} $ can be ordered as an element of  $\mathbb{U}^0_{\shortarrow}$ almost surely by \cite[Theorem 3.1~(iii)]{BBW12}. We denote this element of $\mathbb{U}^0_{\shortarrow}$ by $\mathbf{Z}(\lambda)= ((Y_j^\lambda,N_j^\lambda))_{j\geq 1}$ obtained from $(| \gamma_j^\lambda \big| , N(\gamma_j^\lambda))_{ j \geq 1} $.
\end{remark}
Finally, we define a Markov process $\mathbb{X}:=(\mathbf{X}(s))_{-\infty<s<\infty}$ on $\mathbbm{D}((-\infty,\infty),\ell^2_{\shortarrow})$, called the \emph{multiplicative coalescent process}. 
Think of $\mathbf{X}(s)$ as a collection of masses of some particles (possibly infinite) in a system at time $s$. 
Thus the $i^{th}$ particle has mass $X_i(s)$ at time~$s$. 
The evolution of the system takes place according to the following rule at time $s$: At rate $X_i(s)X_j(s)$,  particles $i$ and $j$ merge into a new particle of mass $X_i(s)+X_j(s)$.
This process has been extensively studied in \cite{A97,AL98}. In particular, Aldous~\cite[Proposition 5]{A97} showed that this is a Feller process.

\section{Main results}
\label{sec_results}
  Consider $n$ vertices labeled by $[n]:=\{1,2,...,n\}$ and a sequence of degrees $\boldsymbol{d} = ( d_i )_{i \in [n]}$ such that $\ell_n = \sum_{i \in [n]}d_i$ is even. For convenience we suppress the dependence of the degree sequence on $n$ in the notation. The configuration model on $n$ vertices with degree sequence $\boldsymbol{d}$ is constructed as follows:
 \begin{itemize}
 \item[] Equip vertex $j$ with $d_{j}$ stubs, or \emph{half-edges}. Two half-edges create an edge once they are paired. Therefore, initially we have $\ell_n=\sum_{i \in [n]}d_i$ half-edges. 
 We pick any one half-edge and pair it with a uniformly chosen half-edge from the remaining unpaired half-edges and keep repeating the above procedure until we exhaust all the unpaired half-edges.
 \end{itemize}
  Note that the graph constructed by the above procedure may contain self-loops or multiple edges. It can be shown \cite[Proposition 7.15]{RGCN1} that, conditionally on $\mathrm{CM}_{n}(\boldsymbol{d})$ being simple, the law of such graphs is uniform over all possible simple graphs with degree sequence $\boldsymbol{d}$.
  \par In this section, we discuss the main results in this paper. As discussed in the introduction, our results are twofold and concern (i) general $\mathrm{CM}_n(\boldsymbol{d})$ at criticality, and (ii) critical percolation on a super-critical configuration model, both under a finite third moment assumption.
\subsection{Configuration model results}
 We consider a sequence of configuration models $(\mathrm{CM}_n(\boldsymbol{d}))_{n\geq 1}$ satisfying the following:
 \begin{assumption} \label{assumption1}
 \normalfont Let $D_n$ denote the degree of a vertex chosen uniformly at random independently of the graph. Then,
 \begin{enumerate}[(i)]
  \item \label{assumption1-1}(\emph{Weak convergence of $D_n$}) \begin{equation} \label{assumption_eq1}
  D_n \dto D
 \end{equation} for some random variable $D$ such that $\mathbbm{E}[D^3] < \infty $.
 \item \label{assumption1-2}(\emph{Uniform integrability of $D_n^3$})
  \begin{equation}
   \label{assumption_eq2}
    \expt{D_n^3}= \frac{1}{n}\sum_{i\in [n]}d_{i}^{3} \to \mathbb{E}\big[ D^{3}\big].
  \end{equation}
  \item \label{assumption1-3}(\emph{Critical window})
  \begin{equation}
   \nu_{n}:= \frac{\sum_{i\in [n]}d_i(d_i-1)}{\sum_{i\in [n]}d_i} =1+\lambda n^{-1/3}+o(n^{-1/3}),
  \end{equation} for some $\lambda\in \mathbbm{R}$.
  \item \label{assumption1-4}$\prob{D=1}>0$.
 \end{enumerate}
 \end{assumption}
  Suppose that $\mathscr{C}_{\sss(1)}$, $\mathscr{C}_{\sss(2)}$,... are the connected components of $\mathrm{CM}_{n}(\boldsymbol{d})$ in decreasing order of size. 
  In case of a tie, order the components according to the values of the minimal indices of vertices in those components. For a connected graph $G$, let $\surp{G}$:= $($number of edges in $G) - (|G| - 1)$ denote the number of surplus edges. Intuitively, this measures the deviation of $G$ from a tree-like structure. Let  $\sigma_{r}= \expt{D^r}$ and consider the reflected Brownian motion, the excursions, and the counting process $\mathbf{N}^\lambda$ as defined in Section \ref{sub_sec_notation} with parameters
 \begin{equation}\label{parameter}
 \mu:=\sigma_1, \quad \eta:= \sigma_{3} \mu - \sigma_{2}^{2},\quad \beta := 1/ \mu.
 \end{equation} Let $\boldsymbol{\gamma}^{\lambda}$ denote the vector of excursion lengths of the process $\mathbf{B}^\lambda_{\mu,\eta}$, arranged in non-increasing order. Our main results  are as follows:
 \begin{theorem}
 \label{thm_main} Fix any $\lambda\in \mathbbm{R}$. Under Assumption~\ref{assumption1},
  \begin{equation}
   n^{-2/3}\big( | \mathscr{C}_{\sss (j)} | \big)_{j\geq 1} \xrightarrow{\mathcal{L}} \boldsymbol{\gamma}^\lambda
  \end{equation}
  with respect to the $\ell^2_{\shortarrow}$ topology.
 \end{theorem}
 Recall the definition of $\mathbf{Z}(\lambda)$ from Remark~\ref{defn_U_0_process}. Order the vector component sizes and surplus edges $ \big( n^{-2/3}\big| \mathscr{C}_{\sss (j)} \big|, \surp{\mathscr{C}_{\sss (j)}} \big)_{j\geq 1}$ as an element of $\mathbb{U}^0_{\shortarrow}$ and denote it by $\mathbf{Z}_n(\lambda)$.
 \begin{theorem} \label{thm_surplus}
  Fix any $\lambda\in \mathbbm{R}$. Under Assumption~\ref{assumption1},
  \begin{equation} \label{eqn_thm_surplus}
  \mathbf{Z}_n(\lambda) \xrightarrow{\mathcal{L}} \mathbf{Z}(\lambda)
  \end{equation} with respect to the $\mathbb{U}^0_{\shortarrow}$ topology.
 \end{theorem}
In words, Theorem~\ref{thm_main} gives the precise asymptotic distribution of the component sizes re-scaled by $n^{2/3}$ and Theorem~\ref{thm_surplus} gives the asymptotic number of surplus edges in each component jointly with their sizes.

\begin{remark}\normalfont The strength of Theorems~\ref{thm_main} and~\ref{thm_surplus} lies in Assumption~\ref{assumption1}. Clearly, Assumption~\ref{assumption1} is satisfied when the distribution of $D$ satisfies an asymptotic power-law relation with finite third moment, i.e., $\PR(D\geq x)\sim x^{-(\tau-1)}(1+o(1))$ for some $\tau >4$. Also, if one has a random degree-sequence that satisfies Assumption~\ref{assumption1} with high probability, then Theorems~\ref{thm_main} and~\ref{thm_surplus} hold conditionally on the degrees. In particular, when the degree sequence consists of an i.i.d sample from a distribution with $\E[D^3]<\infty$ \cite{Jo10}, then Assumption~\ref{assumption1} is satisfied almost surely. We will later see that  degree sequences in the percolation scaling window also satisfy Assumption~\ref{assumption1}. 
\end{remark}

\subsection{Percolation results}
Bond percolation on a graph $G$ refers to deleting edges of $G$ independently with equal probability~$p$. 
In the case $G$ is a random graph, the deletion of edges are also independent of $G$.
 Consider bond percolation on $\mathrm{CM}_n(\boldsymbol{d})$ with probability $p_n$, yielding  $\mathrm{CM}_n(\boldsymbol{d}, p_n)$. We assume the following:
\begin{assumption} \label{assumption2} \normalfont
 \begin{enumerate}[(i)]
  \item \label{assumption2-1} Assumption~\ref{assumption1}~\ref{assumption1-1} and \ref{assumption1-2}  hold for the degree sequence and the  $\mathrm{CM}_n(\boldsymbol{d})$ is super-critical, i.e.
  \begin{equation}
   \nu_n=\frac{\sum_{i\in [n]}d_i(d_i-1)}{\sum_{i\in [n]}d_i}\to \nu= \frac{\expt{D(D-1)}}{\expt{D}} >1.
  \end{equation}
  \item \label{assumption2-2}(\emph{Critical window for percolation}) For some $\lambda\in \mathbbm{R}$,
  \begin{equation}
  p_{n}=p_n(\lambda):=\frac{1}{\nu_n} \bigg( 1+ \frac{\lambda}{n^{1/3} }\bigg).
  \end{equation}
  \end{enumerate}
 \end{assumption}
Note that $p_n(\lambda)$, as defined in Assumption~\ref{assumption2}~\ref{assumption2-2}, is always non-negative for $n$ sufficiently large. 
Now, suppose $\tilde{d}_i\sim \mathrm{Bin}(d_i,\sqrt{p_n})$, $n_+:=\sum_{i\in [n]}(d_i-\tilde{d}_i)$ and $\tilde{n}=n+n_+$.  Consider the degree sequence $\Mtilde{\boldsymbol{d}}$ consisting of $\tilde{d}_i$ for $i\in [n]$ and $n_+$ additional vertices of degree 1, i.e. $\tilde{d}_i=1$ for $i\in [\tilde{n}]\setminus [n]$. 
We will show later that the degree $\tilde{D}_n$ of a random vertex from this degree sequence satisfies Assumption~\ref{assumption1}~\ref{assumption1-1},~\ref{assumption1-2} almost surely for some random variable $\tilde{D}$ with $\mathbbm{E}[\tilde{D}^3]<\infty$.
Moreover, $\tilde{n}/n\to 1+\mu(1-\nu^{-1/2}) = \zeta$ almost surely. 
Now, using the notation in Section~\ref{sub_sec_notation}, define $\tilde{\gamma}_{j}^\lambda = \zeta^{2/3} \bar{\gamma}_j^\lambda$, where $\bar{\gamma}_j^\lambda$ is the $j^{th}$ largest excursion of the inhomogeneous Brownian motion $\mathbf{B}^\lambda_{\mu,\eta}$ with the parameters 
\begin{equation}\label{value-parameters}
\mu=\mathbbm{E}[\tilde{D}],\quad \eta=\mathbbm{E}[\tilde{D}^3]\mathbbm{E}[\tilde{D}]-\mathbbm{E}^2[\tilde{D}^2], \quad \beta=1/\mathbbm{E}[\tilde{D}].
\end{equation} Define the process $\tilde{\mathbf{N}}$ as in \eqref{defn::counting-process}   with the parameter values given by \eqref{value-parameters}. Denote  the $j^{th}$ largest cluster of $\mathrm{CM}_{n}(\boldsymbol{d},p_n(\lambda))$ by  $\mathscr{C}^{p}_{\sss (j)}(\lambda)$. 
Also, let $\mathbf{Z}_n^p(\lambda)$ denote the vector in $\mathbb{U}^0_{\shortarrow}$ obtained by rearranging critical percolation clusters (re-scaled by $n^{2/3}$) and their surplus edges and $\tilde{\mathbf{Z}}(\lambda)$ denote the vector in $\mathbb{U}^0_{\shortarrow}$ obtained by rearranging $( (\sqrt{\nu}| \tilde{\gamma}_j^\lambda | , \tilde{N}(\tilde{\gamma}^\lambda_j)))_{ j\geq 1} $.
    \begin{theorem} \label{thm_percolation}  Under Assumption \ref{assumption2},
    \begin{equation} \label{eqn_thm_percolation}
   \mathbf{Z}_n^p(\lambda) \xrightarrow{\mathcal{L}}  \tilde{\mathbf{Z}}(\lambda)
  \end{equation}  with respect to the $\mathbb{U}^0_{\shortarrow}$ topology.
 \end{theorem}
 Next we consider the percolation clusters for multiple values of $\lambda$. 
 There is a very natural way to couple $(\mathrm{CM}_n(\boldsymbol{d},p_n(\lambda))_{\lambda\in \R}$ described as follows:
Suppose that each edge $(ij)$ of $\mathrm{CM}_n(\boldsymbol{d})$ has an associated i.i.d uniform random variable $U_{ij}$, and the $U_{ij}$'s are also independent of $\mathrm{CM}_n(\boldsymbol{d})$. Now, delete edge $(ij)$ if $U_{ij}>p_{n}(\lambda)$. The obtained graph is distributed as $\mathrm{CM}_n(\boldsymbol{d},p_n(\lambda))$. 
Moreover, if we fix the set of uniform random variables and change $\lambda$, this produces a coupling between the graphs $(\mathrm{CM}_n(\boldsymbol{d},p_n(\lambda))_{\lambda\in \R}$.
 The next theorem shows that the convergence of the component sizes holds jointly in finitely many locations within the critical window, under the above described coupling:
 \begin{theorem} \label{thm_multiple_convergence}
Suppose that 
Assumption \ref{assumption2} holds. Let $\mathbf{C}_n(\lambda)=(n^{-2/3}|\mathscr{C}^p_{\sss (j)}(\lambda)| )_{j\geq 1}$. 
For any $k\geq 1$ and $-\infty<\lambda_0< \lambda_1<\dots<\lambda_{k-1}<\infty$,
 \begin{equation} \label{eqn_thm_multiple_convergence}
 \big( \mathbf{C}_n(\lambda_0), \mathbf{C}_n(\lambda_1), \dots, \mathbf{C}_n(\lambda_{k-1})\big)  \xrightarrow{\mathcal{L}}\sqrt{\nu} (\tilde{\boldsymbol{\gamma}}^{ \lambda_0},\tilde{\boldsymbol{\gamma}}^{ \lambda_1}, \dots, \tilde{\boldsymbol{\gamma}}^{\lambda_{k-1}})
 \end{equation}
with respect to the $(\ell^2_{\shortarrow})^k$ topology. 
 \end{theorem}
\begin{remark}\normalfont \label{rem:mult-coal-heuristics}
The coupling for the limiting process in Theorem~\ref{thm_multiple_convergence} is given by the multiplicative coalescent process described in Section~\ref{sub_sec_notation}. 
This will become more clear when we describe the ideas of the proof. 
To understand this intuitively, notice that the component $\mathscr{C}_{\sss (i)}^p(\lambda)$ consists of some paired half-edges which form the edges of the percolated graph, and some \emph{open half-edges} which were deleted due to percolation.
Denote by $\mathcal{O}_i^p(\lambda)$, the total number of open half-edges of $\mathscr{C}_{\sss (i)}^p(\lambda)$. 
One can think of $\mathcal{O}_i^p$ as the mass of $\mathscr{C}_{\sss (i)}^p$.
Now, as we change the value of the percolation parameter from $p_n(\lambda)$ to $p_n(\lambda+d\lambda)$, exactly one edge is added to the graph and the two endpoints are chosen  proportional to the number of open half-edges of the components  of $\mathrm{CM}_n(\boldsymbol{d},p_n(\lambda))$. 
 By the above heuristics, $\mathscr{C}_{\sss (i)}^p$ and $\mathscr{C}_{\sss (j)}^p$ merge at rate proportional to $\mathcal{O}_i^p\mathcal{O}_j^p$ and creates a component of mass $\mathcal{O}_i^p+\mathcal{O}_j^p-2$. 
Later, we will show that the mass of a component is approximately proportional to the component size. 
Therefore, the component sizes merge \emph{approximately} like the multiplicative coalescent over the critical scaling window.
\end{remark}

\begin{remark}\normalfont  Janson~\cite{J09} studied the phase transition of the maximum component size for percolation on a super-critical configuration model. 
The critical value was shown to be $p=1/\nu$. 
This is precisely the reason behind taking $p_n$ of the form given by  Assumption~\ref{assumption2}~\ref{assumption2-2}. 
The width of the scaling window is intimately related to the asymptotics of the susceptibility function $\sum_i|\mathscr{C}_{\sss (i)}|^2/n$. 
In fact, if $\sum_i|\mathscr{C}_{\sss (i)}|^2\sim n^{1+\eta}$, then the width of the critical window turns out to be $n^{\eta}$ and the largest component sizes are of the order $n^{(1+\eta)/2}$.
This has been universally observed in the random graph literature \cite{A97,NP10b,R12,Jo10,BHL10,DLV14}, even when the scaling limit is not in the same universality class as Erd\H{o}s-R\'enyi random graphs \cite{BHL12,DHLS16} and the same turns out to be the case in this paper.
\end{remark}

\begin{remark}\normalfont Theorem~\ref{thm_main} and Theorem~\ref{thm_surplus} also hold for configuration models conditioned on simplicity. We do not give a proof here. The arguments in \cite[Section 7]{Jo10} can be followed verbatim to obtain a proof of this fact. As a result,  Theorem~\ref{thm_percolation} and Theorem~\ref{thm_multiple_convergence} also hold, conditioned on simplicity.
\end{remark}

The rest of the paper is organized as follows: In Section~\ref{sec_literature_overview}, we give a brief overview of the relevant literature. This will enable the reader  to understand better the relation  of this work  to the large body of literature already present. Also, it will become clear why the choices of the parameters in  Assumption~\ref{assumption1}~\ref{assumption1-3} and  Assumption~\ref{assumption2}~\ref{assumption2-2} should correspond to the critical scaling window. We prove Theorems~\ref{thm_main} and \ref{thm_surplus} in Section~\ref{sec_proof}.  In Section~\ref{sec_vertex} we find the asymptotic degree distribution in each component. This is used along with Theorem~\ref{thm_surplus} to establish Theorem~\ref{thm_percolation} in Section~\ref{sec_percolation}. In Section~\ref{sec_multidimensional}, we analyze the evolution of the component sizes over the percolation critical window and prove Theorem~\ref{thm_multiple_convergence}.

\section{Discussion} \label{sec_discussion}
\subsection{Literature overview}\label{sec_literature_overview}\noindent {\bf Erd\H{o}s-R\'enyi type behavior.} We first explain what `Erd\H{o}s-R\'enyi type behavior' means. The study of \emph{critical window} for random graphs started with the seminal paper \cite{A97} on the Erd\H{o}s-R\'enyi random graphs with $p=n^{-1}(1+\lambda n^{-1/3})$. Aldous showed  in this regime that the largest components are of asymptotic size $n^{2/3}$ and the ordered component sizes (scaled by $n^{2/3}$) asymptotically have the same distribution as the ordered excursion lengths of a Brownian motion with a negative parabolic drift. 
Aldous also considered a natural coupling of the re-scaled vectors of component sizes as $\lambda$ varies, and viewed it as a dynamic $\ell^2_{\shortarrow}$-valued stochastic process. It was shown that the dynamic process can be described by a process called the \emph{standard multiplicative coalescent}, which has the Feller property. This implies the convergence of the component sizes jointly for different $\lambda$ values.
 In Theorem~\ref{thm_multiple_convergence}, we show that similar results hold for the configuration model under a very general set of assumptions. Of course, for general configuration models, there is no obvious way to couple the graphs such that the location parameter in the scaling window varies and percolation seems to be the most natural way to achieve this. By \cite{F07, J09}, percolation on a configuration model can be viewed as a configuration model with a random degree sequence  and this is precisely the reason for studying percolation in this paper.\vspace{.2cm}  \\ 
{\bf Universality and optimal assumptions.} In \cite{BHL10} it was shown that, inside the critical scaling window, the ordered component sizes (scaled by $n^{2/3}$) of an inhomogeneous random graph with \begin{equation}p_{ij}= 1- \exp{\bigg(\frac{-(1+\lambda n^{-1/3})w_iw_j}{\sum_{k \in [n]} w_k} \bigg)}
\end{equation} converge to the ordered excursion lengths of an inhomogeneous Brownian motion with a parabolic drift under only finite third-moment assumption on the weight distribution. We establish a counterpart of this for the configuration model in Theorem~\ref{thm_main}. 
Later Nachmias and Peres~\cite{NP10b} studied the case of percolation scaling window on the random regular graph; for percolation on the configuration model similar results were obtained by Riordan~\cite{R12} for bounded  maximum degrees. 
Joseph~\cite{Jo10} obtained the same scaling limits as Theorem~\ref{thm_main} for the component sizes when the degrees are i.i.d samples from a distribution having finite third moment. 
Theorem~\ref{thm_surplus} and  Theorem~\ref{thm_percolation} prove stronger versions of all these existing results for the configuration model under the optimal assumptions. 
Further, in Theorem~\ref{thm_multiple_convergence}, we give a dynamic picture for percolation cluster sizes in the critical window and show that this dynamics can be approximated by the multiplicative coalescent.\vspace{.2cm} \\
{\bf Comparison to branching processes.} In \cite{MR95,JL09} the phase transition for the component sizes of $\mathrm{CM}_n(\boldsymbol{d})$ was identified in terms of the parameter $\nu=\mathbbm{E}[D(D-1)]/\mathbbm{E}[D]$.
 Janson and Luczak~\cite{JL09} showed that the local neighborhoods of the configuration model can be approximated by a branching process $\mathcal{X}$ which has $\nu$ as its expected progeny and thus, when $\nu >1$, $\mathrm{CM}_n(\boldsymbol{d})$ has a  component $\mathscr{C}_{\max}$ of approximate size $\rho n$, where $\rho$ is the survival probability of $\mathcal{X}$. Further, the progeny distribution of $\mathcal{X}$ has finite variance when $\mathbbm{E}[D^3]<\infty$. Now, for a branching process with mean $\approx 1+\varepsilon$ and finite variance $\sigma^2$, the survival probability is approximately $2\sigma^{-2}\varepsilon$ for small $\varepsilon >0$. This seems to suggest that the largest component size under Assumption~\ref{assumption1} should be of the order $n^{2/3}$ since $\varepsilon=\Theta(n^{-1/3})$. Theorem~\ref{thm_main} mirrors this intuition and shows that in fact all the largest component sizes are of the order $n^{2/3}$.

\subsection{Proof ideas} The proof of Theorem \ref{thm_main} uses standard functional central limit theorem argument. 
Indeed we associate a suitable semi-martingale with the graph obtained from an \emph{exploration algorithm} used to explore the connected components of $\mathrm{CM}_n(\boldsymbol{d})$. 
The martingale part is then shown to converge to an inhomogeneous Brownian motion, and the drift part is shown to converge to a parabola. 
The fact that the component sizes can be expressed in terms of the hitting times of the semi-martingale implies the finite-dimensional convergence of the component sizes. 
The convergence with respect to $\ell^2_{\shortarrow}$ is then concluded using size-biased point process arguments formulated by Aldous~\cite{A97}. 
Theorem~\ref{thm_surplus} requires a careful estimate of the tail probability of the distribution of  surplus edges  when the component size is small and we obtain this using martingale estimates in Lemma~\ref{lem_surplus_delta_bound}.  Theorem~\ref{thm_percolation} is proved by showing that the percolated degree sequence satisfies Assumption~\ref{assumption1} almost surely. 
Finally, we prove Theorem \ref{thm_multiple_convergence} in Section~\ref{sec_multidimensional}.
The key challenges here are that, for each fixed $n$, the components do not merge according to their component sizes, and that the components do not merge exactly like a multiplicative coalescent over the scaling window. 
Thus the main theme of the proof lies in approximating the evolution of the component sizes over the percolation scaling window with a suitable dynamic process that is an exact multiplicative coalescent. 
\subsection{Open problems}
  \begin{enumerate}[(i)]
  \item Theorem~\ref{thm_multiple_convergence} proves the joint convergence at finitely many locations in the scaling window. However, the convergence of $(\mathbf{C}_n(\lambda))_{\lambda\in \mathbbm{R}}$ as a process in $\mathbb{D}((-\infty,\infty),\ell^2_{\shortarrow})$ should also hold provided that one can verify a suitable tightness criterion.
  \item A reason for studying percolation in this paper is to understand the minimal spanning tree of the giant component. For a super-critical configuration model with i.i.d edge weights, it should be the case that the minimal spanning tree can be described by the critically percolated graph at a very high location of the scaling window. 
Such results were obtained in \cite{ABGM13} for the minimal spanning tree on a complete graph. 
The study of minimal spanning trees is still an open question, even for random regular graphs.
    \end{enumerate}

\section{Proofs of Theorems \texorpdfstring{\ref{thm_main}}{1} and \ref{thm_surplus}}
\label{sec_proof}
 \subsection{The exploration process} \label{exploration} Let us explore the graph sequentially using a natural approach outlined in \cite{R12}. At step $k$, divide the set of half-edges into three groups; sleeping half-edges $\mathcal{S}_{k}$, active half-edges $\mathcal{A}_{k}$, and  dead half-edges $\mathcal{D}_{k}$. The depth-first exploration process can be summarized in the following algorithm:
 \begin{algo}[DFS exploration]\label{algo:1}\normalfont At $k=0$, $\mathcal{S}_{k}$ contains all the half-edges and $\mathcal{A}_{k}$, $\mathcal{D}_{k}$ are empty. While ($\mathcal{S}_{k}\neq\varnothing$ or $\mathcal{A}_{k}\neq\varnothing$) we do the following at  stage $k+1$:
 \begin{itemize}
 \item[S1] If $\mathcal{A}_{k}\neq\varnothing$, then take the smallest half-edge $a$ from $\mathcal{A}_{k}$.
 \item[S2] Take the half-edge $b$ from $\mathcal{S}_{k}$ that is paired to $a$. Suppose $b$ is attached to a vertex $w$ (which is necessarily  not discovered yet). Declare $w$ to be discovered,  let $r= d_w-1$ and $b_{w1}, b_{w2}, \dots b_{wr}$ be the half-edges of $w$ other than $b$. Declare $b_{w1}$, $b_{w2}$,..., $b_{wr} ,b$ to be smaller than all other half-edges in $\mathcal{A}_{k}$. 
 Also order the half-edges of $w$ among themselves as $b_{w1} > b_{w2}>\dots  >b_{wr} >b$. 
 Now identify $\mathcal{B}_{k} \subset \mathcal{A}_{k} \cup \{ b_{w1}, b_{w2},\dots, b_{wr} \}$ as the collection of all half-edges in $\mathcal{A}_{k} $ paired to one of the $b_{wi}$'s and the corresponding $b_{wi}$'s. Similarly identify $\mathcal{C}_{k} \subset \{b_{w1}, b_{w2},\dots,b_{wr} \}$ which is the collection of self-loops incident to $w$. Finally, declare $\mathcal{A}_{k+1} = \mathcal{A}_{k} \cup \{b_{w1}, b_{w2},\dots,b_{wr} \} \setminus \big(\mathcal{B}_{k} \cup \mathcal{C}_{k}\big)$, $\mathcal{D}_{k+1} = \mathcal{D}_{k} \cup \{ a,b \} \cup \mathcal{B}_{k} \cup \mathcal{C}_{k}$ and $\mathcal{S}_{k+1}= \mathcal{S}_{k} \setminus \big( \{ b \} \cup \{b_{w1}, b_{w2},...,b_{wr} \} \big)$. Go to stage $k+2$.
  \item[S3] If $\mathcal{A}_{k}= \varnothing$ for some $k$, then take out one half-edge $a$ from $\mathcal{S}_{k}$ uniformly at random and identify the vertex $v$ incident to it. Declare $v$ to be discovered. Let $r= d_v-1$ and assume that $a_{v1}$, $a_{v2}$,..., $a_{vr}$ are the half-edges of $v$ other than $a$ and identify the collection of half-edges involved in self-loops $\mathcal{C}_{k}$ as in Step 2. Order the half-edges of $v$ as $a_{v1}>a_{v2}>\dots>a_{vr}>a$. Set $\mathcal{A}_{k+1} = \{ a, a_{v1}$, $a_{v2}$,..., $a_{vr}\} \setminus \mathcal{C}_{k}$, $\mathcal{D}_{k+1} = \mathcal{D}_{k} \cup \mathcal{C}_{k}$, and $\mathcal{S}_{k+1}= \mathcal{S}_{k} \setminus  \{ a, a_{v1}, a_{v2},...,a_{vr} \} $. Go to stage $k+2$.
 \end{itemize}
 \end{algo}
 In words, we explore a new vertex at each stage and throw away all the half-edges involved in  a loop/multiple edge/cycle  with the vertex set already discovered before proceeding to the next stage.
 The ordering of the half-edges is such that the connected components of $\mathrm{CM}_{n}(\boldsymbol{d})$ are explored in the depth-first way. We call the half-edges of  $\mathcal{B}_{k} \cup \mathcal{C}_{k}$ $cycle$ half-edges because they create loops, cycles or multiple edges in the graph. Let
 \begin{equation}\label{name:c-B-A}
 A_k:=| \mathcal{A}_{k} |,\quad c_{(k+1)}:= (| \mathcal{B}_{k} | +| \mathcal{C}_{k} | )/2 ,\quad  U_k:= | \mathcal{S}_{k} |.
  \end{equation} Let $d_{\sss(j)}$ be the degree of the $j^{th}$ explored vertex and  define the following process:

 \begin{equation}\label{walk1}
 S_{n}(0)=0, \quad S_{n}(i)=\sum_{j=1}^{i}(d_{\sss(j)} -2-2c_{\sss(j)}).
 \end{equation}
 
 The process $\mathbf{S}_{n}= ( S_{n}(i))_{i \in [n]}$ ``encodes the component sizes as lengths of path segments above past minima'' as discussed in \cite{A97}. Suppose $\mathscr{C}_{i}$ is the $i^{th}$ connected component explored by the above exploration process. Define
\begin{equation}
\tau_{k}=\inf \big\{ i:S_{n}(i)=-2k \big\}.
\end{equation}
Then  $\mathscr{C}_{k}$ is discovered between the times $\tau_{k-1}+1$ and $\tau_k$ and  $|\mathscr{C}_{k}|=\tau_{k}-\tau_{k-1}$.

\subsection{Size-biased exploration}
The vertices are explored in a size-biased manner with sizes proportional to their degrees, i.e., if we denote by $v_{\sss(i)}$ the $i^{th}$ explored vertex in Algorithm~\ref{algo:1} and by $d_{\sss(i)}$ the degree of $v_{\sss(i)}$, then
 \begin{equation}
 \mathbbm{P}\big(v_{\sss (i)}=j \vert v_{\sss (1)},v_{\sss (2)},...,v_{\sss (i-1)}\big)=\frac{d_{j}}{\sum_{k\notin \mathscr{V}_{i-1}} d_{k}}= \frac{d_{j}}{\sum_{k \in [n]} d_k - \sum_{k=1}^{i-1} d_{\sss (k)}},\ \forall j\in \mathscr{V}_{i-1},
 \end{equation}
 where $\mathscr{V}_{i}$ denotes the first $i$ vertices to be discovered in the above exploration process.
 The following lemma will be used crucially in the proof of Theorem~\ref{thm_main}:
\begin{lemma} \label{lem_con_1} Suppose that Assumption~\ref{assumption1} holds and denote $\sigma_r = \mathbbm{E} [ D^r ] $ and $\mu = \mathbbm{E} [ D ]$. Then for all $t > 0 $, as $n \to \infty$,
\begin{equation} \label{lem_eq2}
   \sup_{ u\leq t } \Big| n^{-2/3} \sum_{i=1}^{\lfloor n^{2/3} u \rfloor} d_{\sss (i)} - \frac{\sigma_{2}u}{\mu} \Big| \xrightarrow{\mathbbm{P}} 0,
  \end{equation} and
  \begin{equation} \label{lem_eq1}
   \sup_{ u\leq t } \Big| n^{-2/3} \sum_{i=1}^{\lfloor n^{2/3} u \rfloor} d_{\sss (i)}^{2} - \frac{\sigma_{3}u}{\mu} \Big| \xrightarrow{\mathbbm{P}} 0.
  \end{equation}
 \end{lemma}
 The proof of this lemma follows from the two lemmas stated below:
 \begin{lemma}[{\cite[Lemma 8.2]{BSW14}}]\label{lem_general}
Consider a weight sequence $(w_i)_{i\in [n]}$ and let $m=m(n)\leq n$ be increasing with $n$. Let $\{v(i)\}_{i\in[n]}$ be the size-biased reordering of indices $[n]$, where the size of index $i$ is $d_i/\ell_n$. 
Define $\gamma_n = \sum_{i\in[n]}w_id_i/\ell_n$ and  $Y(t) = (m\gamma_n)^{-1}\sum_{i=1}^{\lfloor mt\rfloor} w_{\sss v(i)}$. Further, let $d_{\max} = \max_{i\in [n]}d_i$, and $w_{\max} = \max_{i\in [n]}w_i$. Assume that 
\begin{equation}\label{eq:conditions-size-biased}
 \lim_{n\to\infty}md_{\max}/\ell_n = 0,\quad \text{and}\quad \lim_{n\to\infty} (m\gamma_n)^{-1}w_{\max} = 0.
\end{equation}Then, for any $t>0$, as $n\to\infty$, $
 \sup_{u\leq t}|Y(t)-t|\pto 0.
$
 \end{lemma}
\begin{lemma} \label{lem_d_max} Assumption~\ref{assumption1} implies
 \begin{equation}
  \lim_{k\to \infty}\lim_{n\to \infty} \frac{1}{n} \sum_{j\in [n]}\mathbbm{1}_{\{ d_{j}>k\} }d_{j}^r=0,\quad r=1,2,3.
 \end{equation}
For $r=3$, in particular, this implies $d_{\max}^{3}=o(n)$.
\end{lemma}

\subsection{Estimate of cycle half-edges} The following lemma gives an estimate of the number of cycle half-edges created up to time $t$. This result is proved in \cite{R12} for bounded degrees. In our case, it follows from Lemma~\ref{lem_con_1} as we show below:
\begin{lemma} \label{lem_back_edges}
 For Algorithm~\ref{algo:1}, if $A_k=\big| \mathcal{A}_k \big|$, $ B_k:= \big| \mathcal{B}_k \big|$, and $ C_k:= \big| \mathcal{C}_k \big|$, then
 \begin{equation} \label{lem_back_edges_B}
  \mathbbm{E} \big[ B_k \vert \mathscr{F}_k \big] = (1+o_{\sss \mathbbm{P}}(1))\frac{2A_k}{U_k}  + O_{\sss \mathbbm{P}}(n^{-{2/3}})
 \end{equation} and
 \begin{equation} \label{lem_back_edges_C}
  \mathbbm{E} \big[ C_k \vert \mathscr{F}_k \big] = O_{\sss \mathbbm{P}}(n^{-1})
 \end{equation} uniformly for $k \leq tn^{2/3}$ and any $t >0$, where $\mathscr{F}_k$ is the sigma-field generated by the information revealed up to stage $k$.   Further, all the $O_{\sss\mathbbm{P}} $ and $o_{\sss\mathbbm{P}} $ terms in \eqref{lem_back_edges_B} and \eqref{lem_back_edges_C} can be replaced by $O_{\sss E}$ and $o_{\sss E}$.
 \end{lemma}

 \begin{proof}
  Suppose $ U_k:= \big| \mathcal{S}_k \big|$. First note that by \eqref{lem_eq2}
  \begin{equation}
   \frac{U_k}{n}= \frac{1}{n} \sum_{j \in [n]} d_j - \frac{1}{n} \sum_{j=1}^{k} d_{\sss (j)} = \mathbbm{E} [ D ] + o_{\sss \mathbbm{P}}(1)
  \end{equation} uniformly over $k \leq tn^{2/3}$. 
  Let $a$ be the half-edge that is being explored at stage $k+1$.  
  Now, each of the $(A_k-1)$ half-edges of $\mathcal{A}_k \setminus \{ a \}$ is equally likely to be paired with a half-edge of $v_{\sss (k+1)}$, thus creating two elements of $\mathcal{B}_k$. Also, given $\mathscr{F}_k$ and $v_{\sss (k+1)}$, the probability that a half-edge of $\mathcal{A}_k \setminus \{ a \}$ is paired to one of the half-edges of $v_{\sss (k+1)}$ is $(d_{\sss(k+1)}-1)/(U_{k}-1)$. Therefore,
 \begin{equation}\label{lem_back_edges_eqn2}
 \begin{split}
 \mathbbm{E} \big[ B_k \vert \mathscr{F}_k, v_{\sss(k+1)} \big] = 2(A_k-1)\frac{d_{\sss(k+1)}-1}{U_{k}-1}= 2 \big( d_{\sss(k+1)} -1 \big) \frac{A_k}{U_k-1} - 2 \frac{d_{\sss(k+1)} -1 }{U_k-1}.
 \end{split}
 \end{equation}
 Hence,
 \begin{equation} \label{lem_back_edges_eqn1}
  \mathbbm{E} \big[ B_k \vert \mathscr{F}_k \big] = 2 \mathbbm{E}\big[ d_{\sss(k+1)} -1 \vert \mathscr{F}_k \big] \frac{A_k}{U_k-1} - 2 \frac{\mathbbm{E}\big[ d_{\sss (k+1)} -1 \vert \mathscr{F}_k \big] }{U_k-1}.
 \end{equation}
 Now, using \eqref{lem_eq2} and \eqref{lem_eq1},
  \begin{equation}
 \mathbbm{E}\big[ d_{\sss(k+1)} -1 \vert \mathscr{F}_k \big] = \frac{\sum_{j \notin \mathscr{V}_k} d_j (d_j-1)}{\sum_{j \notin \mathscr{V}_k} d_j}  = \frac{\sum_{j \in [n]}d_j^2}{\sum_{j \in [n]}d_j}-1+o_{\sss\mathbbm{P}}(1) = 1+o_{\sss\mathbbm{P}}(1).
 \end{equation} uniformly over $k \leq tn^{2/3}$, where the last step follows from Assumption~\ref{assumption1}~\ref{assumption1-3}. 
 Further, using the fact $\PR(D=1)>0$, $U_k\geq c_0n $ for some constant $c_0>0$ uniformly over $k\leq tn^{2/3}$.
 Thus, \eqref{lem_back_edges_eqn1} gives \eqref{lem_back_edges_B}.  The fact that all the $O_{\sss\mathbbm{P}}$, $o_{\sss\mathbbm{P}}$ can be replaced by $O_{\sss E}$, $o_{\sss E}$ follows from $\sum_{j\in [n]}d_j^r-kd_{\max}^r \leq \sum_{j\notin\mathscr{V}_k}d_j^r\leq  \sum_{j\in [n]}d_j^r$ for $r=1,2$, together with $d_{\max}=o(n^{1/3})$. To prove \eqref{lem_back_edges_C}, note that
 \begin{equation}
  \mathbbm{E} \big[ C_k \vert \mathscr{F}_k, v_{\sss (k+1)} \big] = 2(d_{\sss (k+1)}-2)\frac{d_{\sss (k+1)}-1}{U_{k}-1}.
 \end{equation}
 By Assumption \ref{assumption1} and \eqref{lem_eq2} 
 \begin{equation}\label{expt-dk2-F-k}\mathbbm{E}[d_{\sss(k+1)}^2|\mathscr{F}_k] = \frac{\sum_{j\notin \mathscr{V}_k} d_j^3}{\sum_{j\notin \mathscr{V}_k} d_j} \leq \frac{\sum_{j\in [n]} d_j^3}{\sum_{j\in [n]} d_j+o_{\sss \mathbbm{P}}(n^{2/3})} = O_{\sss\mathbbm{P}}(1),
 \end{equation}uniformly for $k\leq tn^{2/3}$. Therefore,
 \begin{equation}\label{lem_back_edges_eqn3}
 \begin{split}
  \mathbbm{E} \big[ C_k \vert \mathscr{F}_k \big] = O_{\sss\mathbbm{P}}(n^{-1})
  \end{split}
 \end{equation} uniformly over   $k \leq tn^{2/3}$. Again, $O_{\sss\mathbbm{P}}$ term can be replaced by $O_{\sss E}$, as argued before.
 \end{proof}

 \subsection{Key ingredients} 
 For any $\mathbbm{D}[0,\infty)$-valued process $\mathbf{X}_n$ define $\bar{X}_n(u):= n^{-1/3} X_n( \lfloor n^{2/3}u \rfloor )$ and  $ \bar{\mathbf{X}}_{n}:=( \bar{X}_n(u) )_{u \geq 0} $. The following result is the main ingredient for proving Theorem~\ref{thm_main}. Recall the definition of $\mathbf{B}^\lambda_{\mu,\eta}$ from \eqref{def:inhomogen:BM} with parameters given in \eqref{parameter}. 
 \begin{theorem}[Convergence of the exploration process] \label{thm_main1}
 Under Assumption~\ref{assumption1}, as $n \to \infty$,
  \begin{equation}
   \bar{\mathbf{S}}_{n} \xrightarrow{\mathcal{L}} \mathbf{B}^\lambda_{\mu,\eta}
  \end{equation}
 with respect to the $Skorohod$ $J_{1}$ topology.
 \end{theorem}
  As in \cite{Jo10}, we will prove this by approximating  $\mathbf{S}_{n}$ by a simpler process defined as
\begin{equation} \label{walk2} 
  s_{n}(0)=0, \quad s_{n}(i)=\sum_{j=1}^{i}(d_{\sss(j)} -2).
 \end{equation}
 Note that the difference between the processes $\mathbf{S}_n$ and $\mathbf{s}_n$ is due to the cycles, loops, and multiple-edges encountered during the exploration.
Following the approach of \cite{Jo10}, it will be enough to prove the following:
 \begin{proposition} \label{thm_main2} Under Assumption~\ref{assumption1}, as $n \to \infty$,
  \begin{equation}
   \bar{\mathbf{s}}_{n} \xrightarrow{\mathcal{L}} \mathbf{B}^\lambda_{\mu,\eta}
  \end{equation}
  with respect to  the $Skorohod$ $J_{1}$ topology.
\end{proposition}
 \begin{remark} \label{remark_cont_cadlag} \normalfont It will be shown that the distributions of $\bar{\mathbf{S}}_n$ and $\bar{\mathbf{s}}_n$ are very close as $n\to\infty$, and therefore, Proposition~\ref{thm_main2} implies Theorem \ref{thm_main1}. This is achieved by proving that we will not see \emph{too many} cycle half-edges up to the time $\lfloor n^{2/3}u\rfloor$ for any  fixed $u >0$.
 \end{remark}
From here onwards we will look at the continuous versions of the processes $\bar{\mathbf{S}}_{n}$ and $\bar{\mathbf{s}}_{n}$ by linearly interpolating between the values at the jump points and write it using the same notation. It is easy to see that these continuous versions differ from their c\`adl\`ag versions by at most $n^{-1/3}d_{\max}=o(1)$ uniformly on $[0,T]$, for any $T>0$. Therefore, the convergence in law of the continuous versions implies the convergence in law of the c\`adl\`ag versions and vice versa.
Before proceeding to show that Theorem~\ref{thm_main1} is a consequences of Proposition~\ref{thm_main2}, we will need to bound the difference of these two processes in a suitable way. We need the following lemma.  Recall the definition of $c_{\sss (k+1)}:= (B_k+C_k)/2$ from \eqref{name:c-B-A}.
\begin{lemma} \label{cor_c_k}
 Fix $t >0$ and $M>0$ (large). Define
 $E_{n}(t,M):=\big\{\max_{s\leq t}\{\bar{s}_{n}(s)-\min_{u\leq s} \bar{s}_{n}(u)\} < M \big\}.$ Then
 \begin{equation} \label{cor_c_k_eqn}
  \limsup\limits_{n \to \infty} \sum_{k \leq tn^{2/3}} \mathbbm{E} \big[ c_{\sss(k)} \mathbbm{1}_{E_{n}(t,M)} \big] < \infty.
 \end{equation}
 \end{lemma}

 \begin{proof} Lemma~\ref{cor_c_k} is similar to \cite[Lemma 6.1]{Jo10}. We add a brief proof here.  Note that, for all large $n$, $A_k \leq Mn^{1/3}$ on $E_{n}(t,M)$, because
  \begin{equation}
   A_k = S_n(k)-\min\limits_{j \leq k} S_n(j)= s_n(k) - 2 \sum_{j=1}^{k} c_{\sss(j)} - \min\limits_{j \leq k} S_n(j) \leq s_n(k)-\min\limits_{j \leq k} s_n(j),
  \end{equation} where the last step follows by noting that $\min_{j \leq k} s_n(j) \leq \min_{j \leq k} S_n(j) + 2 \sum_{j=1}^{k} c_{\sss(j)}$.  By Lemma~\ref{lem_back_edges},
  \begin{equation}
  \mathbbm{E} \big[ c_{\sss(k)} \mathbbm{1}_{E_{n}(t,M)} \big]  \leq \frac{Mn^{1/3}}{\mu n} + o(n^{-2/3})= \frac{M}{\mu}n^{-2/3} + o(n^{-2/3})
  \end{equation} uniformly for $k \leq tn^{2/3}$. 
  Summing over $1 \leq k \leq tn^{2/3}$ and taking the $\limsup$ completes the proof.
 \end{proof}
The proof of the fact that Theorem~\ref{thm_main1} follows from Proposition~\ref{thm_main2} and Lemma~\ref{cor_c_k} is standard (see \cite[Section 6.2]{Jo10}) and we skip the proof for the sake of brevity. 
  From here onward the main focus of this section will be to prove Proposition~\ref{thm_main2}. We use the martingale functional central limit theorem in a similar manner as~\cite{A97}.
  \begin{proof}[Proof of Proposition~\ref{thm_main2}] Let $ \{\mathscr{F}_{i} \}_{i \geq 1} $ be the natural filtration defined in Lemma \ref{lem_back_edges}. Recall the definition of $s_n(i)$ from \eqref{walk2}. By the Doob-Meyer decomposition \cite[Theorem 4.10]{KS91} we can write
  \begin{equation}\label{split_up}
    s_{n}(i) = M_{n}(i)+A_{n}(i),\quad s_{n}^{2}(i) = H_{n}(i)+B_{n}(i),
  \end{equation}
where
 \begin{subequations}
  \begin{equation} \label{defn_martingale}
   M_{n}(i)= \sum_{j=1}^{i} \big(d_{\sss(j)}-\mathbbm{E} \big[ d_{\sss(j)}\vert \mathscr{F}_{j-1} \big] \big),
  \end{equation}
  \begin{equation}
   A_{n}(i)= \sum_{j=1}^{i} \mathbbm{E}\big[d_{\sss(j)}-2 \vert \mathscr{F}_{j-1} \big],
  \end{equation}
  \begin{equation}
   B_{n}(i)= \sum_{j=1}^{i} \big( \mathbbm{E} \big[ d_{\sss(j)}^{2}\vert \mathscr{F}_{j-1} \big]-\mathbbm{E}^{2} \big[ d_{\sss(j)}\vert \mathscr{F}_{j-1} \big] \big).
  \end{equation}
 \end{subequations}

Recall that for a discrete time stochastic process $(X_n(i))_{i\geq 1}$, we denote $\bar{X}_n(t)=n^{-1/3}X_n(\lfloor tn^{2/3}\rfloor )$. 
Our result follows from the martingale functional central limit theorem  \cite[Theorem 2.1]{W07} if we can prove the following four conditions: For any $u >0$,
\begin{subequations}
 \begin{equation}\label{condition1}
  \sup_{s\leq u}\big| \bar{A}_{n}(s)-\lambda s+\frac{\eta s^2}{2\mu^{3}}\big| \xrightarrow{\mathbbm{P}} 0,
 \end{equation}
 \begin{equation} \label{condition2}
  n^{-1/3}\bar{B}_{n}(u)\xrightarrow{\mathbbm{P}} \frac{\eta}{\mu^{2}} u,
 \end{equation}
 \begin{equation} \label{condition3}
  \mathbbm{E}\big[\sup_{s\leq u}\big| \bar{M}_{n}(s)-\bar{M}_{n}(s-)\big| ^{2}\big] \to 0,
 \end{equation}
 and
 \begin{equation} \label{condition4}
  n^{-1/3}\mathbbm{E}\big[\sup_{s\leq u}\vert \bar{B}_{n}(s)-\bar{B}_{n}(s-)\vert\big] \to 0.
 \end{equation}
\end{subequations}
\par Indeed \eqref{condition1} gives rise to the quadratic drift term of the limiting distribution. Conditions \eqref{condition2}, \eqref{condition3}, \eqref{condition4} are the same as \cite[Theorem 2.1, Condition (ii)]{W07}. The facts that the jumps of both the martingale and the quadratic-variation process go to zero and that the quadratic variation process is converging to the quadratic variation of an inhomogeneous Brownian Motion, together imply the convergence of the martingale term. The validation of these conditions are given separately in the subsequent part of this section.
\end{proof}

\begin{lemma} The conditions \eqref{condition2}, \eqref{condition3}, and \eqref{condition4} hold.
\end{lemma}

\begin{proof}
 Denote by $\sigma_{r}(n)=\frac{1}{n} \sum_{i\in [n]}d_{i}^{r},\: r=2,3$ and $\mu(n)=\frac{1}{n} \sum_{i\in [n]}d_{i}$. To prove \eqref{condition2}, it is enough to prove that
 \begin{equation}
   n^{-2/3}B_{n}(\lfloor un^{2/3}\rfloor) \xrightarrow{\mathbbm{P}}  \frac{\sigma_3 \mu - \sigma_{2}^{2}}{\mu^{2}} u.
 \end{equation}
 Recall that $\mathbbm{E}[ d_{\sss(i)}^{2}\vert \mathscr{F}_{i-1} ] = \sum_{j \notin \mathscr{V}_{i-1}}d_{j}^{3}/\sum_{j \notin \mathscr{V}_{i-1}}d_{j}.$  
 Further, uniformly over $i \leq un^{2/3}$,
 \begin{equation}\label{sum-deg-explored}
 \sum_{j \notin \mathscr{V}_{i-1}}d_{j} = \sum_{j \in [n]}d_{j} + \OP(d_{\max}i) = \ell_n+\oP(n).
 \end{equation}
  Assume that, without loss of generality, $j\mapsto d_{j}$ is non-increasing.  Then, uniformly over $i \leq un^{2/3}$,
 \begin{equation} \label{eq_lem_1}
 \bigg| \sum_{j \notin \mathscr{V}_{i-1}}d_{j}^{3} - n\sigma_{3}(n) \bigg| \leq \sum_{j=1}^{un^{2/3}} d_{j}^{3}.
 \end{equation}
 For each fixed $k$,
 \begin{equation}
 \frac{1}{n} \sum_{j=1}^{un^{2/3}} d_{j}^{3} \leq \frac{1}{n}\sum_{j=1}^{un^{2/3}} \mathbbm{1}_{ \{ d_{j} \leq k \} } d_{j}^{3} + \frac{1}{n}\sum_{j \in [n]} \mathbbm{1}_{ \{ d_{j} > k \} } d_{j}^{3} \leq k^{3}un^{-1/3} + \frac{1}{n}\sum_{j \in [n]} \mathbbm{1}_{ \{ d_{j} > k \} } d_{j}^{3} =o(1),
 \end{equation}
 where we first let $n \to \infty$ and then $k \to \infty$ and use Lemma \ref{lem_d_max}. Therefore, the right-hand side of \eqref{eq_lem_1} is $o(n)$ and we conclude that, uniformly over $i \leq un^{2/3}$,
 \begin{equation}
 \mathbbm{E}\big[d_{\sss(i)}^{2}\vert \mathscr{F}_{i-1}\big] = \frac{\sigma_{3}}{\mu} + \oP(1).
 \end{equation}
A similar argument gives
  \begin{equation}
   \mathbbm{E}\big[d_{\sss(i)} \vert \mathscr{F}_{i-1}\big] = \frac{\sigma_{2}}{\mu} + \oP(1),
  \end{equation}
and \eqref{condition2} follows by noting that the error term is $\oP(1)$, uniformly over $i \leq un^{2/3}$.
 The proofs of \eqref{condition3} and \eqref{condition4} are rather short and we present them below. For \eqref{condition3}, we bound
 \begin{align}
   \mathbbm{E} \Big[ \sup_{s \leq u} \vert \bar{M}_{n}(s) - \bar{M}_{n}(s -) \vert^{2} \Big]&= n^{-2/3} \mathbbm{E} \Big[ \sup_{k \leq un^{2/3}} \vert M_{n}(k) - M_{n}(k-1) \vert^{2} \Big] \nonumber\\
    & = n^{-2/3} \mathbbm{E} \Big[ \sup_{k \leq un^{2/3}} \big| d_{\sss(k)} - \mathbbm{E}[ d_{\sss(k)} \vert \mathscr{F}_{k-1}] \big|^{2} \Big]\nonumber\\
    & \leq n^{-2/3} \mathbbm{E} \Big[ \sup_{k \leq un^{2/3}}  d_{\sss(k)}^{2} \Big] + n^{-2/3} \mathbbm{E} \Big[ \sup_{k \leq un^{2/3}} \mathbbm{E}^2\big[ d_{\sss(k)} \vert \mathscr{F}_{k-1} \big] \Big] \nonumber\\
    & \leq 2n^{-2/3}d_{\max}^2. 
    \end{align}
Similarly,  \eqref{condition4} gives
 \begin{align}
n^{-1/3} \mathbbm{E} \big[ \sup_{s \leq u} \vert \bar{B}_{n}(s) - \bar{B}_{n}(s -) \vert^{2} \big]&= n^{-2/3} \mathbbm{E} \big[ \sup_{k \leq un^{2/3}} \vert B_{n}(k) - B_{n}(k-1) \vert \big] \hspace{2cm}\nonumber \\
& = n^{-2/3} \mathbbm{E} \big[ \sup_{k \leq un^{2/3}} \mathrm{var} \big( d_{\sss(k)} \vert \mathscr{F}_{k-1} \big) \big]\\& \leq 2n^{-2/3} d_{\max}^{2},\nonumber
 \end{align} and Conditions \eqref{condition3} and \eqref{condition4} follow from Lemma \ref{lem_d_max} using $d_{\max}=o(n^{1/3})$.
 \end{proof}
 Next, we prove Condition \eqref{condition1} which requires some more work. Note that
\begin{align}\label{the-split-up}
 &\mathbbm{E} \big[ d_{\sss(i)} -2 \vert \mathscr{F}_{i-1} \big]  = \frac{\sum_{j \notin \mathscr{V}_{i-1}} d_{j}(d_{j}-2)}{\sum_{j \notin \mathscr{V}_{i-1}} d_{j}}\nonumber\\
 &\hspace{2cm}= \frac{\sum_{j \in [n]} d_{j}(d_{j}-2)}{\sum_{j \in [n]} d_{j}}- \frac{\sum_{j \in \mathscr{V}_{i-1}} d_{j}(d_{j}-2)}{\sum_{j \in [n]} d_{j}} + \frac{\sum_{j \notin \mathscr{V}_{i-1}} d_{j}(d_{j}-2)\sum_{j \in \mathscr{V}_{i-1}} d_{j} }{\sum_{j \notin \mathscr{V}_{i-1}} d_{j}\sum_{j \in [n]} d_{j}} \nonumber\\
 &\hspace{2cm}= \frac{\lambda}{n^{1/3}} - \frac{\sum_{j \in \mathscr{V}_{i-1}} d_{j}^{2}}{\sum_{j \in [n]} d_{j}} + \frac{\sum_{j \notin \mathscr{V}_{i-1}} d_{j}^{2}\sum_{j \in \mathscr{V}_{i-1}} d_{j} }{\sum_{j \notin \mathscr{V}_{i-1}} d_{j}\sum_{j \in [n]} d_{j}} +o(n^{-1/3}),
\end{align} where the last step follows from Assumption~\ref{assumption1}~\ref{assumption1-3}.
Therefore,
\begin{equation} \label{eq_cond_1}
\begin{split}
A_{n}(k) &= \sum_{i=1}^{k}\mathbbm{E} \big[ d_{\sss(i)} -2 \vert \mathscr{F}_{i-1} \big] \\
&= \frac{k\lambda}{n^{1/3}} - \sum_{i=1}^{k} \frac{\sum_{j \in \mathscr{V}_{i-1}} d_{j}^{2}}{\sum_{j \in [n]} d_{j}} + \sum_{i=1}^{k}\frac{\sum_{j \notin \mathscr{V}_{i-1}} d_{j}^{2}\sum_{j \in \mathscr{V}_{i-1}} d_{j} }{\sum_{j \notin \mathscr{V}_{i-1}} d_{j}\sum_{j \in [n]} d_{j}} + o(kn^{-1/3}).
\end{split}
\end{equation}

 The following lemma estimates the sums on the right-hand side of \eqref{eq_cond_1}:

 \begin{lemma} \label{lem_con_2}
  For all $u>0$, as $n \to \infty$,
  \begin{equation} \label{lem_eq3}
   \sup_{s \leq u} \bigg| n^{-1/3} \sum_{i=1}^{\lfloor s n^{2/3} \rfloor} \sum_{j=1}^{i-1} \frac{d_{\sss(j)}^{2}}{\ell_n}- \frac{\sigma_{3} s^{2}}{2\mu^{2}} \bigg| \xrightarrow{\mathbbm{P}} 0
  \end{equation} and
  \begin{equation} \label{lem_eq4}
   \sup_{s \leq u} \bigg| n^{-1/3} \sum_{i=1}^{\lfloor s n^{2/3} \rfloor} \sum_{j=1}^{i-1} \frac{d_{\sss(j)}}{\ell_n}- \frac{\sigma_{2} s^{2}}{2\mu^{2}} \bigg| \xrightarrow{\mathbbm{P}} 0.
  \end{equation} Consequently, 
  \begin{equation} \label{lem:estimate-drift}
\sup_{s \leq u} \bigg|n^{-1/3}\sum_{i=1}^{\lfloor s n^{2/3} \rfloor}\frac{\sum_{j \notin \mathscr{V}_{i-1}} d_{j}^{2}\sum_{j \in \mathscr{V}_{i-1}} d_{j} }{\sum_{j \notin \mathscr{V}_{i-1}} d_{j}\sum_{j \in [n]} d_{j}} - \frac{\sigma_{2}^{2}s^{2}}{2\mu^{3}}\bigg| \xrightarrow{\mathbbm{P}} 0.
\end{equation}
 \end{lemma}
 \begin{proof}
  Notice that
   \begin{equation}\begin{split}
     &\sup_{s \leq u} \Big| n^{-1/3} \sum_{i=1}^{\lfloor s n^{2/3} \rfloor} \sum_{j=1}^{i-1} \frac{d_{\sss (j)}^{2}}{\ell_n}- \frac{\sigma_{3} s^{2}}{2\mu^{2}} \Big| = \sup_{k \leq un^{2/3}} \Big| n^{-1/3} \sum_{i=1}^{k} \sum_{j=1}^{i-1} \frac{d_{\sss (j)}^{2}}{\ell_n}- \frac{\sigma_{3} k^{2}}{2\mu^{2}n^{4/3}} \Big| \\
 & \leq \frac{1}{\ell_{n}} \sup_{k \leq un^{2/3}} \Big| n^{-1/3} \sum_{i=1}^{k} \Big(\sum_{j=1}^{i-1} d_{\sss (j)}^{2}- \frac{\sigma_{3} (i-1)}{\mu} \Big) \Big| \\
 &\hspace{1cm}+ \sup_{k \leq un^{2/3}} \Big| \frac{k \sigma_3}{2 \mu \ell_n n^{1/3}} \Big| + \sup_{k \leq un^{2/3}} \Big| \frac{k^2 \sigma_3}{2 \mu \ell_n n^{1/3}} -\frac{k^2 \sigma_3}{2 \mu^2 n^{4/3}} \Big|\\
 & \leq \frac{1}{\ell_n} n^{-1/3} un^{2/3} \sup_{i \leq un^{2/3}} \Big| \sum_{j=1}^{i}   d_{\sss (j)}^{2}- \frac{\sigma_{3} i}{\mu}  \Big| + o(1)+ \frac{\sigma_3 n^{-1/3}}{2 \mu} \Big| \frac{1}{\ell_n} - \frac{1}{n \mu} \Big| u^2 n^{4/3}\\
 & \leq \frac{u}{\mu + o(1)} \sup_{s \leq u} \Big| \Big(n^{-2/3}\sum_{j=1}^{\lfloor s n^{2/3} \rfloor} d_{\sss(j)}^{2}- \frac{\sigma_{3} s}{\mu} \Big) \Big| +o(1).
 \end{split}
   \end{equation} and \eqref{lem_eq3} follows from \eqref{lem_eq1} in Lemma~\ref{lem_con_1}. The proof of \eqref{lem_eq4} is similar and it follows from~\eqref{lem_eq2}.
 We now show \eqref{lem:estimate-drift}. Recall that $\sigma_2(n) = \frac{ 1}{n}\sum_{i \in [n]} d_i^2$ and observe
\begin{equation}
 \frac{1}{n}\sum_{j \notin \mathscr{V}_{i-1}} d_{j}^{2} = \sigma_{2}(n)-\frac{1}{n} \sum_{j \in \mathscr{V}_{i-1}} d_{j}^{2}= \sigma_2(n) +o_{\sss\mathbbm{P}}(1)
\end{equation} uniformly over $i \leq un^{2/3}$ where we use Lemma \ref{lem_con_1} to conclude the uniformity. Similarly, \eqref{sum-deg-explored} implies that $\sum_{j \notin \mathscr{V}_{i-1}} d_{j} = \ell_n+o_{\sss\mathbbm{P}}(n)$ uniformly over $i \leq un^{2/3}$. Therefore,
\begin{equation}
n^{-1/3}\sum_{i=1}^{k}\frac{\sum_{j \notin \mathscr{V}_{i-1}} d_{j}^{2}\sum_{j \in \mathscr{V}_{i-1}} d_{j} }{\sum_{j \notin \mathscr{V}_{i-1}} d_{j}\sum_{j \in [n]} d_{j}} = \frac{n\sigma_{2}(n)+o_{\sss\mathbbm{P}}(n)}{\ell_n+o_{\sss\mathbbm{P}}(n)} n^{-1/3}\sum_{i=1}^{k}\frac{\sum_{j \in \mathscr{V}_{i-1}} d_{j}}{\ell_n}
\end{equation}
 and Assumption~\ref{assumption1}, combined with \eqref{lem_eq4}, complete the proof.
\end{proof}
 \begin{lemma}
 Condition \eqref{condition1} holds.
 \end{lemma}
 \begin{proof}
  The proof follows by using  Lemma \ref{lem_con_2} in  \eqref{eq_cond_1}.
 \end{proof}
 \subsection{Finite dimensional convergence of the ordered component sizes} Note that the convergence of the exploration process in Theorem~\ref{thm_main1} implies that, for any large $T>0$, the $k$-largest components explored up to time $Tn^{2/3}$ converge to the  $k$-largest excursions above past minima of $\mathbf{B}^\lambda_{\mu,\eta}$ up to time $T$. Therefore, we can conclude the finite dimensional convergence of the ordered components sizes in the whole graph if we can show that the large components are explored \emph{early} by the exploration process. The following lemma formalizes the above statement:
 \begin{lemma}\label{lem:large-com-explored-early}Let $\mathscr{C}_{\max}^{\sss \geq T}$ denote the largest component which is started exploring after time $Tn^{2/3}$ in Algorithm~\ref{algo:1}. Then, for any $\delta >0$,
 \begin{equation}\label{large-com-explored-early}
  \lim_{T\to\infty}\limsup_{n\to\infty}\prob{|\mathscr{C}_{\max}^{\sss \geq T}|>\delta n^{2/3}}=0.
 \end{equation}
 \end{lemma}Let us first state the two main ingredients to complete the proof of Lemma~\ref{lem:large-com-explored-early}: 
 \begin{lemma}[{\cite[Lemma 5.2]{J09b}}]\label{lem:janson-lemma}  Consider $\mathrm{CM}_n(\boldsymbol{d})$ with $\nu_n<1$ and let $\mathscr{C}(V_n)$ denote the component containing the vertex $V_n$, where $V_n$ is a vertex chosen uniformly at random independently of the graph~$\mathrm{CM}_n(\boldsymbol{d})$. Then,
 \begin{equation}
  \expt{|\mathscr{C}(V_n)|}\leq 1+\frac{\expt{D_n}}{1-\nu_n}.
 \end{equation}  
 \end{lemma}

 \begin{lemma}\label{lem-time-nu-rel}Define, $\nu_{n,i}=\sum_{j\notin \mathscr{V}_{i-1}}d_j(d_j-1)/\sum_{j\notin \mathscr{V}_{i-1}}d_j.$ There exists  some constant $C_0>0$ such that for any $T>0$,
 \begin{equation}\label{nu-n-i:nu-n:relation}
 \nu_{n,Tn^{2/3}}= \nu_n-C_0 Tn^{-1/3}+o_{\sss\mathbbm{P}}(n^{-1/3}).
\end{equation}
 \end{lemma}
 \begin{proof}
  Using a similar split up as in \eqref{the-split-up}, we have
\begin{equation}\label{nu-n-i:nu-n:split}
 \nu_{n,i}= \nu_n+\frac{\sum_{j \in \mathscr{V}_{i-1}}d_j(d_j-1)}{\ell_n}- \frac{\sum_{j\notin \mathscr{V}_{i-1}}d_j(d_j-1)\sum_{j\in \mathscr{V}_{i-1}}d_j}{\ell_n\sum_{j\notin \mathscr{V}_{i-1}}d_j}.
\end{equation}Now, \eqref{lem_eq2} and \eqref{lem_eq1} give that, uniformly over $i\leq Tn^{2/3}$,\
\begin{subequations}
\begin{equation}
 \frac{\sum_{j\notin \mathscr{V}_{i-1}}d_j(d_j-1)}{\sum_{j\notin \mathscr{V}_{i-1}}d_j}= \frac{\sum_{j\in [n]}d_j(d_j-1)+o_{\sss\mathbbm{P}}(n^{2/3})}{\sum_{j\in [n]}d_j+o_{\sss\mathbbm{P}}(n^{2/3})}=1+o_{\sss\mathbbm{P}}(n^{-1/3}),
\end{equation}
\begin{equation}
 \sum_{j\in \mathscr{V}_{i-1}}d_j(d_j-2)=\Big(\frac{\sigma_3}{\mu}-2\Big)(i-1)+o_{\sss \mathbbm{P}}(n^{2/3}).
\end{equation}
\end{subequations}Further, note that $\sigma_3-2\mu = \mathbbm{E}[D(D-1)(D-2)]+ \mathbbm{E}[D(D-2)]>0$, by Assumption~\ref{assumption1}~\ref{assumption1-3}, and~\ref{assumption1-4}. Therefore, \eqref{nu-n-i:nu-n:split} gives \eqref{nu-n-i:nu-n:relation}.
 \end{proof}
 \begin{proof}[Proof of Lemma~\ref{lem:large-com-explored-early}] Let $i_{\sss T}:=\inf\{i\geq Tn^{2/3}: S_n(i)= \inf_{j\leq i}S_n(j)\} $. Thus, $i_{\sss T}$ denotes the first time we finish exploring a component after time $Tn^{2/3}$. Note that, conditional on the explored vertices up to time $i_{\sss T}$, the remaining graph $\bar{\mathcal{G}}$ is still a configuration model. Let $\bar{\nu}_n=\sum_{i\in \bar{\mathcal{G}}}d_i(d_i-1)/\sum_{i\in \bar{\mathcal{G}}}d_i$ be the criticality parameter of $\bar{\mathcal{G}}$. Then, using \eqref{nu-n-i:nu-n:relation}, we can conclude that 
 \begin{equation}\label{eq:bar-nu-n}
 \bar{\nu}_n\leq \nu_n- C_0Tn^{-1/3}+\oP(n^{-1/3}).
 \end{equation} Take $T>0$ such that $\lambda -C_0T <0$. Thus, with high probability, $\bar{\nu}_n<1$. 
 Denote the component corresponding to a randomly chosen vertex from $\bar{\mathcal{G}}$ by $\mathscr{C}^{\sss \geq T}(V_n)$, and the $i^{\sss th}$ largest component of $\bar{\mathcal{G}}$ by $\mathscr{C}_{\sss (i)}^{\sss \geq T}$. 
 Also, let $\bar{\PR}$ denote the probability measure conditioned on $\mathscr{F}_{i_{\sss T}}$, and let $\bar{\E}$ denote the corresponding expectation. 
 Now, for any $\delta >0$,
 \begin{equation}
  \begin{split} 
  &\bar{\PR}\bigg( \sum_{i\geq 1}|\mathscr{C}_{\sss (i)}^{\sss \geq T}|^2>\delta^2 n^{4/3}  \bigg)\leq \frac{1}{\delta^2 n^{4/3}}\sum_{i\geq 1}\bar{\E}\big( |\mathscr{C}^{\sss \geq T}_{\sss(i)}|^2\big)\\
  &\hspace{4.25cm}\leq  \frac{1}{\delta^2 n^{1/3}}\bar{\E}\big( |\mathscr{C}^{\sss \geq T}(V_n)|\big)\leq \frac{1}{\delta^2(-\lambda+C_0T+\oP(1))},
  \end{split}
\end{equation} where the second step follows from the Markov inequality and the last step follows by combining Lemma~\ref{lem:janson-lemma} and \eqref{eq:bar-nu-n}. Noting that $\bar{\nu}_n<1$ with high probability, we get
\begin{equation}
 \limsup_{n\to\infty}\prob{|\mathscr{C}_{\max}^{\sss \geq T}|>\delta n^{2/3}}\leq \frac{C}{\delta^2 T},
\end{equation} for some constant $C>0$ and large $T>0$ and the proof follows.
 \end{proof}
 \begin{theorem}\label{thm:conv-fd-comp}
  The convergence in Theorem~\ref{thm_main} holds with respect to the product topology.
 \end{theorem}
 \begin{proof}
  The proof follows from Theorem~\ref{thm_main1} and Lemma~\ref{lem:large-com-explored-early}.
 \end{proof}
 \subsection{Proof of Theorem~\ref{thm_main}}  \label{sec_l2_tightness}
 The proof of Theorem~\ref{thm_main} follows  using similar argument as \cite[Section 3.3]{A97}. 
 However, the proof is a bit tricky since the components are explored in a size-biased manner with sizes being the total degree in the components (not the component sizes as in \citep{A97}). 
 For a sequence of random variables $\mathbf{Y}=(Y_i)_{i\geq 1}$ satisfying $\sum_{i\geq 1}Y_i^2<\infty$ almost surely, define $\boldsymbol{\xi}:=(\xi_i)_{i\geq 1}$ such that $\xi_i|\mathbf{Y}\sim \mathrm{Exp}(Y_i)$ and the coordinates of $\boldsymbol{\xi}$ are independent conditional on $\mathbf{Y}$.  
 For $a\geq 0$, let $\mathscr{S}(a):=\sum_{\xi_i\leq a}Y_i$. Then the \emph{size biased point process} is defined to be the random collection of points $\Xi:=\{(\mathscr{S}(\xi_i),Y_i)\}_{i\geq 1}$ (see \cite[Section 3.3]{A97}).  
 We will use Lemma 8, Lemma 14 and Proposition 15 from \cite{A97}. 
 Let $\mathfrak{C}:=\{\mathscr{C}: \mathscr{C}\text{ is a component of }\mathrm{CM}_n(\boldsymbol{d})\}$. 
 Consider the collection $\boldsymbol{\xi}:=(\xi(\mathscr{C}))_{\mathscr{C}\in \mathfrak{C}}$ such that conditional on $(\sum_{k\in \mathscr{C}}d_k, |\mathscr{C}|)_{\mathscr{C}\in \mathfrak{C}}$, $\xi(\mathscr{C})$ has an exponential distribution with rate $n^{-2/3}\sum_{k\in \mathscr{C}}d_k$ independently over $\mathscr{C}$. Then the order in which Algorithm~\ref{algo:1} explores the components can be obtained by ordering the components according to their $\xi$-value.  Recall that $\mathscr{C}_i$ denotes the $i^{th}$ explored component by Algorithm~\ref{algo:1} and let $D_i:=\sum_{k\in\mathscr{C}_i}d_k$. Define the size biased point process
 \begin{equation}\Xi_n:=\Big(n^{-2/3} \sum_{j=1}^{i} D_i  , \hspace{.2cm}n^{-2/3} D_i \Big)_{i \geq 1}. 
 \end{equation}  Also define the point processes
 \begin{equation} \Xi_{n}^{'} := \Big( n^{-2/3} \sum_{j=1}^i \big| \mathscr{C}_j \big|, \hspace{.2cm} n^{-2/3} \big| \mathscr{C}_i\big|  \Big)_{i \geq 1}, \quad \Xi_{\infty} := \big\{ \big( l(\gamma), |\gamma| \big):\text{ }\gamma \text{ an excursion of } \mathbf{B}^\lambda_{\mu,\eta} \big\}, 
 \end{equation} where we recall that $l(\gamma)$ are the left endpoints of the excursions of $\mathbf{B}^\lambda_{\mu,\eta}$ and $|\gamma|$ is the length of the excursion $\gamma$ (see \eqref{defn::reflected-BM}). Note that $\Xi_n'$ is not a size biased point process. However, applying \cite[Lemma 8]{A97} and Theorem \ref{thm_main1}, we get $ \Xi_{n}^{'} \xrightarrow{\sss \mathcal{L}} \Xi_{\infty}$.  We claim that 
 \begin{equation}\label{conv:SBPP}
 \Xi_n \xrightarrow{\mathcal{L}} 2 \Xi_{\infty}.
 \end{equation}
 To verify the claim, note that \eqref{lem_eq2}  and Assumption~\ref{assumption1}~\ref{assumption1-3} together imply, for any $t>0$,
 \begin{equation}\label{eqn:deg-close-comp}
  \sup_{u \leq t} \big| n^{-2/3} \sum_{i=1}^{\lfloor un^{2/3} \rfloor} d_{\sss(i)} - \frac{\sigma_2}{\mu}u \big|
  = \sup_{u \leq t} \big| n^{-2/3} \sum_{i=1}^{\lfloor un^{2/3} \rfloor} d_{\sss(i)} - 2u \big| \xrightarrow{\mathbbm{P}} 0,
 \end{equation}since $\sigma_2/\mu =\E[D^2]/\E[D]=2$. 
  Thus, \eqref{conv:SBPP} follows using \eqref{eqn:deg-close-comp}.
 Now, the point process $2 \Xi_{\infty}$ satisfies all the conditions of \cite[Proposition 15]{A97} as shown by Aldous. 
 Thus, \cite[Lemma 14]{A97} gives
  \begin{align} \label{eqn_tightness}
	\big\{ D_{\sss (i)}\big\}_{i \geq 1}\text{ is tight in } \ell^2_{\shortarrow}.
  \end{align} 
This implies that $\big( n^{-2/3} \big| \mathscr{C}_{\sss (i)} \big| \big)_{i \geq 1}$ is tight in $\ell^2_{\shortarrow}$ by simply observing that $|\mathscr{C}_i|\leq \sum_{k\in\mathscr{C}_i}d_k+1$. Therefore, the proof of Theorem~\ref{thm_main} is complete using Theorem~\ref{thm:conv-fd-comp}.
\qed

\subsection{Proof of Theorem~\ref{thm_surplus}} \label{sec_surplus_edges}
The proof of Theorem~\ref{thm_surplus} is completed in two separate lemmas. 
In Lemma~\ref{lem:surp:poisson-conv} we first show that the convergence in Theorem~\ref{thm_surplus} holds with respect to the $\ell^2_{\shortarrow}\times \mathbb{N}^{\infty}$ topology. The tightness of $(\mathbf{Z}_n)_{n\geq 1}$ with respect to the $\mathbb{U}^0_{\shortarrow}$ topology is ensured in Lemma~\ref{sufficient-U0-conv-condn}.

 \begin{lemma} \label{lem:surp:poisson-conv} Let $N_n^\lambda(k)$ be the number of surplus edges discovered up to time $k$ and $\bar{N}^\lambda_n(u) = N_n^\lambda(\lfloor un^{2/3} \rfloor)$. Then, as $n\to\infty$,
 \begin{equation}\bar{\mathbf{N}}_n^\lambda\dto \mathbf{N}^\lambda,
 \end{equation} where $\mathbf{N}^\lambda$ is defined in \eqref{defn::counting-process}.
 \end{lemma}
 \begin{proof}
 Recall the definitions of $a$, $b$, $\mathcal{A}_k$, $\mathcal{B}_k$, $\mathcal{C}_k$, $\mathcal{S}_k$ from Section \ref{exploration}. 
 Recall also that $ A_k:= \big| \mathcal{A}_k \big|$, $ B_k:= \big| \mathcal{B}_k \big|$, $ C_k:= \big| \mathcal{C}_k \big|$, $ U_k:= \big| \mathcal{S}_k \big|$, $c_{(k+1)}:= (\big| \mathcal{B}_{k} \big| +\big| \mathcal{C}_{k} \big|)/2 $ from Section~\ref{exploration}. Notice that $A_k=S_n(k)-\min_{j \leq k} S_n(j)$. 
 From Lemma~\ref{lem_back_edges}, we can conclude that, uniformly over $k\leq un^{2/3}$,
 \begin{equation} \label{eqn_surplus_intensity}
  \mathbbm{E} \big[ c_{(k+1)} \vert \mathscr{F}_k \big] = \frac{A_k}{\mu n}+ O_{\sss\mathbbm{P}}(n^{-1}).
 \end{equation}
 The counting process $\mathbf{N}_n^\lambda$ has conditional intensity (conditioned on $\mathscr{F}_{k -1}$) given by \eqref{eqn_surplus_intensity}. Writing the conditional intensity in  \eqref{eqn_surplus_intensity} in terms of $\bar{\mathbf{S}}_n$, we get that the conditional intensity of the re-scaled process $\bar{\mathbf{N}}^\lambda_n$ is given by 
 \begin{equation} \label{rate:surplus:scaled}
 \frac{1}{\mu} [\bar{S}_n(u)-\min_{\tilde{u} \leq u} \bar{S}_n(\tilde{u})]+ o_{\sss \PR}(1).
 \end{equation} 
 Denote by  $\bar{W}_n(u):=\bar{S}_n(u)-\min_{\tilde{u} \leq u} \bar{S}_n(\tilde{u})$ which is  the reflected version $\bar{\mathbf{S}}_n$. By Theorem~\ref{thm_main},  
 \begin{equation}\bar{\mathbf{W}}_n\dto\mathbf{W}^\lambda,
 \end{equation} where $\mathbf{W}^\lambda$ is defined in \eqref{defn::reflected-BM}. 
 Therefore, we can assume that there exists a probability space such that $\bar{\mathbf{W}}_n\to\mathbf{W}^\lambda$ almost surely. Using \cite[Theorem 1; Chapter 5.3]{LS89}, and the continuity of the sample paths of $\mathbf{W}^\lambda$, we conclude the proof.
\end{proof} 
\begin{lemma}\label{sufficient-U0-conv-condn} The vector $(\mathbf{Z}_n)_{n\geq 1}$ is tight with respect to the $\mathbb{U}^0_{\shortarrow}$ topology.
\end{lemma}
The proof of Lemma~\ref{sufficient-U0-conv-condn} makes use of the following crucial estimate of the probability that a component with small size has very large number of surplus edges:
\begin{lemma} \label{lem_surplus_delta_bound}
Assume that $\lambda <0.$ Let $V_n$ denote a vertex chosen uniformly at random, independent of the graph $\mathrm{CM}_n(\boldsymbol{d})$ and let $\mathscr{C}(V_n)$ denote the component containing $V_n$.  Let $\delta_k=\delta k^{-0.12}$. Then, for $\delta > 0$ (small),
\begin{equation}
 \prob{\surp{\mathscr{C}(V_n)}\geq K,|\mathscr{C}(V_n)|\in (\delta_K n^{2/3},2\delta_Kn^{2/3})}\leq \frac{C\sqrt{\delta}}{n^{1/3}K^{1.1}},
\end{equation}
 where $C$ is a fixed constant independent of $n,\delta, K$.
\end{lemma}
\begin{proof}[Proof of Lemma~\ref{sufficient-U0-conv-condn}]
To simplify the notation, we write $Y_i^n=n^{-2/3} |\mathscr{C}_{\sss (i)}|$ and $N_i^n=$\# $\{$surplus edges in $\mathscr{C}_{\sss(i)}\}$. Let $Y_i$, $N_i$ denote the distributional limits of $Y_i^n$ and $N_i^n$ respectively. Recall from Remark~\ref{defn_U_0_process} that $\mathbf{Z}(\lambda)$ is almost surely $\mathbb{U}^0_{\shortarrow}$-valued. Using Lemma~\ref{lem:surp:poisson-conv} and the definition of $d_{\mathbb{U}}$ from \eqref{defn_U_metric}, the proof of Lemma~\ref{sufficient-U0-conv-condn} is complete if we can show that, for any $\eta >0$
 \begin{equation} \label{eqn_sufficient_for_U_0_convergence}
 \lim_{\varepsilon\to 0}\limsup_{n\to\infty}\PR\bigg( \sum_{Y_i^n\leq \varepsilon} Y_i^n N_i^n> \eta \bigg)=0.
 \end{equation} 
First, consider the case $\lambda <0$. For every $\eta,\varepsilon >0$ sufficiently small 
\begin{align}
  &\PR\bigg( \sum_{Y_i^n\leq \varepsilon} Y_i^n N_i^n> \eta \bigg)\leq \frac{1}{\eta}\E \bigg[\sum_{i=1}^{\infty}Y_i^n N_i^n \1_{\{ Y_i^n\leq \varepsilon\}} \bigg]= \frac{n^{-2/3}}{\eta} \E \bigg[\sum_{i=1}^{\infty}|\mathscr{C}_{\sss(i)}| N_i^n \1_{\{ |\mathscr{C}_{\sss(i)}|\leq \varepsilon n^{2/3}\}} \bigg]\nonumber\\
  &= \frac{n^{1/3}}{\eta}\expt{\mathrm{SP}(\mathscr{C}(V_n))\1_{\{ |\mathscr{C}(V_n)|\leq \varepsilon n^{2/3}\}}}\nonumber\\
  &= \frac{n^{1/3}}{\eta}\sum_{k=1}^{\infty}\sum_{i\geq \log_2(1/(k^{0.12}\varepsilon))}\PR\bigg(\mathrm{SP}(\mathscr{C}(V_n))\geq k, |\mathscr{C}(V_n)|\in \bigg(\frac{n^{2/3}}{2^{i+1}k^{0.12}},\frac{n^{2/3}}{2^{i}k^{0.12}} \bigg] \bigg)\nonumber\\
  &\leq \frac{C}{\eta} \sum_{k=1}^{\infty}\frac{1}{k^{1.1}}\sum_{i\geq \log_2(1/(k^{0.12}\varepsilon))} 2^{-(1/2)i} \leq \frac{C}{\eta}\sum_{k=1}^{\infty}\frac{\sqrt{\varepsilon}}{k^{1.04}}  =O(\sqrt{\varepsilon}),
 \end{align}
  where we have used Lemma~\ref{lem_surplus_delta_bound}. Therefore, \eqref{eqn_sufficient_for_U_0_convergence} holds when $\lambda <0$. Now consider the case $\lambda >0$.  For $T>0$ (large), let \begin{equation}
 \mathcal{K}_n:=\{i: Y_i^n\leq \varepsilon, \mathscr{C}_{\sss (i)} \text{ is explored before }Tn^{2/3}\}.
\end{equation}Then, by applying the Cauchy-Schwarz inequality,
\begin{equation}\label{eq:SP-C-T}
\begin{split}
\sum_{i\in \mathcal{K}_n}Y_i^nN_i^n&\leq \Big( \sum_{i\in \mathcal{K}_n}(Y_i^n)^2\Big)^{1/2}\times  \Big( \sum_{i\in \mathcal{K}_n}(N_i^n)^2\Big)^{1/2}\\
&\leq \Big( \sum_{i\in \mathcal{K}_n}(Y_i^n)^2\Big)^{1/2}\times (\# \text{ surplus edges  explored before }Tn^{2/3})
\end{split}
\end{equation}
For the case $\lambda >0$, we can use similar ideas as the proof of Lemma~\ref{lem:large-com-explored-early}, i.e., we can run the exploration process till $Tn^{2/3}$ and the unexplored graph becomes a configuration model with negative criticality parameter for large $T>0$, by \eqref{nu-n-i:nu-n:relation}.  Thus, the proof can be completed using \eqref{eq:SP-C-T}, the $\ell^{2}_{\shortarrow}$ convergence of the component sizes given by Theorem~\ref{thm_main} and Lemma~\ref{lem:surp:poisson-conv}, and the proof for the case $\lambda<0$.
\end{proof}
\begin{proof}[Proof of Lemma~\ref{lem_surplus_delta_bound}] 
To complete the proof of Lemma~\ref{lem_surplus_delta_bound}, we will use martingale techniques coupled with Lemma~\ref{lem:janson-lemma}. Fix $\delta > 0$ (small). First we describe another way of exploring $\mathscr{C}(V_n)$ which turns out to be convenient to work with.
 \begin{algo}[Exploring $\mathscr{C}(V_n)$]\label{algo:2}\normalfont Consider the following exploration of $\mathscr{C}(V_n)$: \begin{itemize}
 \item[(S0)] Initialize all half-edges to be alive. Choose a vertex from $[n]$ uniformly at random and declare all its half-edges active.
 \item[(S1)] In the next step, take any active half-edge and pair it uniformly with another alive half-edge. Kill these paired half-edges. Declare all the half-edges corresponding to the new vertex (if any) active. Keep repeating (S1) until the set of active half-edges is empty.
\end{itemize}
\end{algo}Unlike Algorithm~\ref{algo:1}, we need not see a new vertex at each stage and we explore only two half-edges at each stage. In this proof, $\mathscr{F}_l$ denotes the sigma-field containing information revealed up to stage~$l$ by Algorithm~\ref{algo:2} and $\mathscr{V}_l$ denotes the vertex set discovered up to time $l$. Recall that we denote by $D_n$ the degree of $V_n$. Define the exploration process~$\mathbf{s}_n'$ by,
\begin{equation}
 s_n'(0)=D_n,\ s_n'(l)= \sum_{i\in [n]} d_i\mathcal{I}_i^n(l)-2l,
\end{equation} where $\mathcal{I}_i^n(l)= \ind{i\in \mathscr{V}_l}$. 
Therefore, $s_n'(l)$ counts the number of active half-edges at time $l$, until $\mathscr{C}(V_n)$ is explored. Note that $\mathscr{C}(V_n)$ is explored when $\mathbf{s}'_n$ hits zero and the hitting time to zero gives the number of edges in $\mathscr{C}(V_n)$, since exactly one edge is being explored at each time step.
We will use a generic constant $C$ to denote a positive constant that can be different in different equations.  For $H>0$,  let \begin{equation} \label{defn:gamma}
\gamma := \inf \{ l\geq 1: s_n'(l)\geq H \text{ or }  s_n'(l)= 0 \}\wedge 2\delta n^{2/3}.
\end{equation} Note that
\begin{equation}\label{exploration:super_martingale}
\begin{split}
 \expt{s_n'(l+1)-s_n'(l)\vert \mathscr{F}_l}&= \sum_{i\in [n]}d_i\prob{i\in \mathscr{V}_{l+1}\vert \mathscr{F}_l,\mathcal{I}_i^n(l) = 0} -2\\
 &= \frac{ \sum_{i\notin \mathscr{V}_l}d_i^2}{\ell_n-2l-1}-2\leq \frac{ \sum_{i\in [n]}d_i^2}{\ell_n-2l-1}-2\\
 & =\frac{\lambda}{n^{1/3}}+o(n^{-1/3})+\frac{2l+1}{\ell_n-2l-1}\times \frac{\sum_{i\in [n]}d_i^2}{\ell_n}   \leq 0
\end{split}
\end{equation} uniformly over $l\leq 2\delta n^{2/3}$ for all small $\delta >0$ and large $n$, where the last step follows from the fact that $\lambda<0$. Therefore, $\{s_n'(l)\}_{l= 1}^{2\delta n^{2/3}}$ is a super-martingale. The optional stopping theorem now implies
  \begin{equation}
   \mathbbm{E}\left[D_n\right] \geq \mathbbm{E}\left[s_n'(\gamma)\right] \geq H \mathbbm{P}\left( s'_n(\gamma) \geq H \right).
  \end{equation} Thus,
  \begin{equation} \label{eqn::bound_geq_H_at_stopping_time}
    \mathbbm{P}\left( s'_n(\gamma) \geq H \right) \leq \frac{\expt{D_n}}{H}.
  \end{equation}
We put $H=n^{1/3}K^{1.1}/\sqrt{\delta}$. To simplify the notation, we write $s_n'[0,t]\in A$ to denote that $s_n'(l)\in A,$ for all $ l\in [0,t]$.  Notice that, for $K\geq 1$,
 \begin{equation}\label{surp:sup:less}\begin{split}
  &\prob{\surp{\mathscr{C}(V_n)}\geq K,|\mathscr{C}(V_n)|\in (\delta_K n^{2/3},2\delta_Kn^{2/3})}\\
  &\leq \prob{s_n'(\gamma)\geq H}+\prob{\surp{\mathscr{C}(V_n)}\geq K, s_n'[0,2\delta_K n^{2/3}]< H, s_n'[0,\delta_K n^{2/3}]>0}.
  \end{split}
 \end{equation}
 Here we have used the fact that if there is at least one surplus edge in $\mathscr{C}(V_n)$, the number of edges in $\mathscr{C}(V_n)$ is at least $\mathscr{C}(V_n)$. Therefore, $|\mathscr{C}(V_n)|>\delta_Kn^{2/3}$ implies $s_n'[0,\delta_K n^{2/3}]>0$.
 Let us denote the event that surplus edges appear at times  $l_1,\dots,l_K$,  $s_n'[0,2\delta_K n^{2/3}]< H$, and $s_n'[0,\delta_K n^{2/3}]>0$ by $\mathrm{SPB}(l_1,\dots,l_K)$.
   Now,
 \begin{equation}
 \begin{split}
   &\prob{\surp{\mathscr{C}(V_n)}\geq K, s_n'[0,2\delta_K n^{2/3}]< H, s_n'[0,\delta_K n^{2/3}]>0}\\
  &\hspace{.5cm}\leq \sum_{1\leq l_1<\dots< l_K\leq 2\delta_K n^{2/3}} \prob{\mathrm{SPB}(l_1,\dots,l_K)}\\
  &\hspace{.5cm}=\sum_{1\leq l_1<\dots<l_K\leq 2\delta_K n^{2/3}}\expt{\ind{0<s_n'[0,l_K-1]<H, \mathbf{SP}(l_K-1)=K-1}Y},
 \end{split}
 \end{equation}
 where
 \begin{align}
  Y&=\prob{K^{th}\text{ surplus occurs at time }l_K,  s_n'[l_K,2\delta_K n^{2/3}]< H, s_n'[l_K,\gamma]>0\mid \mathscr{F}_{l_K-1} }\nonumber\\
  &\leq \frac{CK^{1.1}n^{1/3}}{\ell_n\sqrt{\delta}}\leq \frac{CK^{1.1}}{n^{2/3}\sqrt{\delta}}.
 \end{align}
 Therefore, using induction, 
 \begin{equation}\label{exploration:bounded:surplus}
 \begin{split}
  &\prob{\surp{\mathscr{C}(V_n)}\geq K, s_n'[0,2\delta_K n^{2/3}]< H, s_n'[0,\delta_K n^{2/3}]>0}\\
  &\hspace{1cm}\leq C\bigg( \frac{K^{1.1}}{\sqrt{\delta}n^{2/3}}\bigg)^K\frac{(2\delta n^{2/3})^{K-1}}{K^{0.12(K-1)}(K-1)!}\sum_{l_1=1}^{2\delta_K n^{2/3}}\prob{|\mathscr{C}(V_n)|\geq l_1}\\
  &\hspace{1cm}\leq C \frac{\delta^{K/2}}{K^{1.1}n^{2/3}}  \expt{|\mathscr{C}(V_n)|},
  \end{split}
 \end{equation}where we have used the fact that $\#\{1\leq l_2<\dots<l_k\leq 2\delta n^{2/3}\}\leq(2\delta n^{2/3})^{K-1}/(K-1)!$ and have used the Stirling approximation for $(K-1)!$ in the last step. Since $\lambda <0$, we can use Lemma~\ref{lem:janson-lemma} to conclude that for all sufficiently large $n$
 \begin{equation} \label{expectation:random:vert:comp}
  \expt{|\mathscr{C}(V_n)|}\leq Cn^{1/3},
 \end{equation} for some constant $C>0$ and we get the desired bound for \eqref{surp:sup:less}.
  The proof of Lemma~\ref{lem_surplus_delta_bound} is now complete by applying \eqref{eqn::bound_geq_H_at_stopping_time}~and~\eqref{exploration:bounded:surplus} in \eqref{surp:sup:less}.
\end{proof}

\section{Vertices of degree $k$} \label{sec_vertex}
In this section, we compute the number of vertices of degree $k$ in each connected component at criticality. This will be useful in  Section~\ref{sec_percolation} and \ref{sec_multidimensional}. Such an estimate was proved in \cite[Theorem 2.4]{JL09} for supercritical graphs under stronger moment assumptions.
\begin{lemma}\label{lem:sec:vertex} Denote by $N_k(t)$ the number of vertices of degree $k$ discovered up to time $t$. For any $t>0$, uniformly over $k$,
  \begin{equation}
   \sup\limits_{u \leq t} \big| n^{-2/3} N_k(un^{2/3})-\frac{kn_k}{\ell_n} u \big| =  O_{\sss\mathbbm{P}}((kn^{1/3})^{-1}).
  \end{equation}
 \end{lemma}
 \begin{proof}
  By setting $w_i=\mathbbm{1}_{ \{ d_i=k \} }$ in Lemma~\ref{lem_general}  we can directly conclude that 
  \begin{equation} \sup\limits_{u \leq t} \big| n^{-2/3} N_k(un^{2/3})-\frac{kn_k}{\ell_n} u \big|\xrightarrow{\mathbbm{P}}0.
  \end{equation} However, one can repeat the same arguments leading to the proof of Lemma~\ref{lem_general} and obtain that
  \begin{equation} \label{lem_martingale_eqn}
  \mathbbm{P} \Big( \sup\limits_{u \leq t} \Big| n^{-2/3} N_k(un^{2/3})-\frac{kn_k}{\ell_n} u \Big| > \frac{A}{kn^{1/3}} \Big)  \leq \frac{3\Big( k^3 s^2 \frac{r_k}{( \mathbbm{E}[D] )^2 } + \sqrt{  s \frac{k^3r_k}{\mathbbm{E}[D]}} \Big)}{A}+ o(1).
  \end{equation} Now, we can use the finite third-moment assumption to conclude that the numerator in the right hand side can be taken to be uniform over $k$. Thus, the proof follows.
 \end{proof}
  Define $v_k(G):=$ the number of vertices of degree $k$ in the connected graph $G$.
As a corollary to Lemma~\ref{lem:sec:vertex} and  \eqref{large-com-explored-early}, we can deduce  that
  \begin{equation}\label{eqn_vertices_of_degree_k-ord}
   v_k \big( \mathscr{C}_{\sss(j)} \big) = \frac{kr_k}{\mathbbm{E}[D]} \big| \mathscr{C}_{\sss (j)} \big| +O_{\sss\mathbbm{P}}\big((k^{-1}n^{1/3})\big).
  \end{equation}Moreover, the following also holds: Let $\mathrm{ord}(\boldsymbol{x})$ denote the vector with elements of $\boldsymbol{x}$ ordered in a non-increasing manner. 
\begin{lemma} \label{vertices_of_degree_k_ord}
For each $k\geq 1$ denote by $\mathbf{V}_k^n:=(n^{-2/3}v_k ( \mathscr{C}_j ))_{j\geq 1}$. Then, $\{\mathrm{ord}(\mathbf{V}_k^n)\}_{n\geq 1}$ is tight in $\ell^2_{\shortarrow}$.
\end{lemma}
\begin{proof}
 Note that for any $j\geq 1$, $v_k(\mathscr{C}_{\sss(j)})\leq |\mathscr{C}_{\sss(j)}|$ uniformly over $k$. The proof now follows from \eqref{eqn_vertices_of_degree_k-ord} and $\ell^2_{\shortarrow}$ tightness of the component sizes given in Theorem~\ref{thm_main}.
\end{proof}

\begin{remark}\normalfont Define $\mathbf{V}^n:=(n^{-2/3}v_k(\mathscr{C}_j))_{k,j\geq 1}$. Then $\{\mathrm{ord}(\mathbf{V}^n)\}_{n\geq 1}$ is also tight in $\ell^2_{\shortarrow}$.
\end{remark}

\section{Critical Percolation}
\label{sec_percolation}

\subsection{Percolation on Configuration Model} \label{percolation_defn}
Let $p = p_n \in (0,1)$ be the percolation parameter. Recall the notation $\mathrm{CM}_{n}(\boldsymbol{d},p)$ for the random graph obtained after deleting edges of $\mathrm{CM}_{n}(\boldsymbol{d})$ independently with probability $1-p$. Suppose, $\boldsymbol{d}'$ is the random degree sequence obtained after percolation. 
Fountoulakis~\cite{F07} showed that, given $\boldsymbol{d}'$, the law of $\mathrm{CM}_{n}(\boldsymbol{d},p)$  is same as the law of $\mathrm{CM}_{n}(\boldsymbol{d}')$.  We will use the following construction of $\mathrm{CM}_{n}(\boldsymbol{d},p)$ due to Janson~\cite{J09}:
\begin{algo} \label{algo:3}
\normalfont \begin{itemize}
 \item[(S1)] For each half-edge $e$, let $v_e$ be the vertex to which $e$ is attached. With probability $1-\sqrt{p}$, one detaches $e$ from $v_e$ and associates $e$ to a new vertex $v'$. Color the new vertex $red$. This is done independently for every existing half-edge. Let $n_+$ be the number of red vertices created and $\tilde{n}=n+n_+$.  Suppose, $\Mtilde{\boldsymbol{d}} = ( \tilde{d}_i )_{i \in [\tilde{n}]}$ is the new degree sequence obtained by the above procedure, i.e. $\tilde{d}_i \sim \text{Bin} (d_i, \sqrt{p})$ for $i \in [n]$ and $\tilde{d}_i=1$ for $i \in [\tilde{n}] \setminus [n]$.
 \item[(S2)] Construct $\mathrm{CM}_{\tilde{n}}(\Mtilde{\boldsymbol{d}})$, independently of (S1).
 \item[(S3)] Delete all the red vertices.
 \end{itemize}
 \end{algo}
\begin{remark}\label{rem:alt-s3}
 \normalfont
 It was argued in \cite{J09} that the obtained multigraph also has the same distribution as $\mathrm{CM}_{n}(\boldsymbol{d},p)$ if we replace (S3) by
 \begin{itemize}
 \item[(S3$'$)] Instead of deleting red vertices, choose any $n_+$ degree one vertices uniformly at random, independently of (S1) and (S2), and delete them.
 \end{itemize}
\end{remark}
\begin{remark}\normalfont The construction of $\mathrm{CM}_{\tilde{n}}(\Mtilde{\boldsymbol{d}})$ in Algorithm~\ref{algo:3} consists of two stages of randomization, the first one is described by (S1), and the second one by (S2). We will consider the following probability space to describe the randomization arising from Algorithm~\ref{algo:3}~(S1): Suppose we have a sequence of degree sequences $(\boldsymbol{d})_{n\geq 1}$. Let  $\mathbbm{P}_p^n$ denote the  probability measure induced on $\mathbb{N}^{\infty}$ by Algorithm~\ref{algo:3}~(S1). Denote the product measure of $(\PR_p^n)_{n\geq 1}$ by $\mathbbm{P}_p$. Thus (S1) is performed independently on $\boldsymbol{d}=\boldsymbol{d}(n)$ as $n$ varies.  All the almost sure statements in this section will be with  respect to the probability measure $\mathbbm{P}_p$. 
\end{remark}
\begin{remark} \label{remark-perc} \normalfont
The idea of the proof of Theorem~\ref{thm_percolation} is as follows.  We show that $\Mtilde{\boldsymbol{d}}$, under Assumption~\ref{assumption2}, satisfies Assumption~\ref{assumption1} $\mathbbm{P}_p$ almost surely and then estimate the number of vertices to be deleted from each component using Lemma~\ref{lem:sec:vertex}. 
Since deleting a degree one vertex does not break up any component, we can just subtract this from the component sizes of $\mathrm{CM}_{\tilde{n}}(\Mtilde{\boldsymbol{d}})$ to get the component sizes of $\mathrm{CM}_{n}(\boldsymbol{d},p_n(\lambda))$. 
Since the degree one vertices do not get involved in surplus edges, deleting degree one vertices does not change the number of surplus edges.
\end{remark}
\subsection{Proof of Theorem~\ref{thm_percolation}}
We now consider the critical window corresponding to percolation. The goal is to prove Theorem~\ref{thm_percolation}. Let $n_j$ and $\tilde{n}_j$ be the number of vertices of degree $j$ before and after performing Algorithm~\ref{algo:3}~(S1) respectively. Further let 
\begin{equation}\tilde{\nu}_n = \frac{\sum_{i \in [\tilde{n}]}\tilde{d}_{i}\big( \tilde{d}_i-1 \big)}{\sum_{i \in [\tilde{n}]}\tilde{d}_{i}}.
\end{equation} 
For convenience we write $r_j=\mathbbm{P}(D=j)$. Denote by $\tilde{n}_{jl}$, the number of vertices that had degree $l$ before and have degree $j$ after performing Algorithm~\ref{algo:3}~(S1). Therefore, $\tilde{n}_{jl} \sim \text{Bin}\big( n_l,b_{lj}(\sqrt{p_n}) \big) $, where $b_{lj}(\sqrt{p_n})= \binom{l}{j} (\sqrt{p_n})^j (1-\sqrt{p_n})^{l-j}$. Using the strong law of large numbers for triangular arrays, note that $\mathbbm{P}_p$ almost surely,
$\tilde{n}_{jl} = n_l b_{lj}(\sqrt{p_n}) + o(n_l) = nr_l b_{lj}(\sqrt{p_n}) + o(n_l).$  Now, $\sum_{l\geq 1}|n_l/n-r_l|\to 0$ and therefore, for all $j \geq 2$,  $\mathbbm{P}_p$ almost surely
 \begin{equation}\label{estimate:n-j-tilde}
  \frac{\tilde{n}_j}{n} = \frac{\sum_{l=j}^{\infty}\tilde{n}_{jl}}{n} = \sum_{l=j}^{\infty}r_lb_{lj}(\sqrt{p}_n)+ o(1).
 \end{equation}
 Also, $n_+ = \sum_{i \in [n]}\big( d_i - \tilde{d}_i \big) \sim \text{Bin}(\ell_n, 1-\sqrt{p_n})$. Therefore, using the similar arguments as  \eqref{estimate:n-j-tilde} again, $\mathbbm{P}_p$ almost surely,
 \begin{equation} \label{estimate_n+}
 \begin{split}
  \frac{n_+}{n} & = \mathbbm{E}(D) \big( 1- \sqrt{p_n} \big) + o(1),  			
  \end{split}
 \end{equation}
\begin{equation} \label{estimate_n_1_tilde}
\frac{\tilde{n}_{1}}{n} =\frac{\sum_{l=1}^{\infty}\tilde{n}_{1l}+n_+}{n}= \frac{\sum_{l=1}^{\infty}\tilde{n}_{1l}}{n}+\mathbbm{E}(D) \big( 1- \sqrt{p_n} \big) + o(1),
\end{equation}
and
 \begin{equation}\label{limit-n-tilde}
  \frac{\tilde{n}}{n} =1+\frac{n_+}{n}=1+\mathbbm{E}(D) \big( 1- \sqrt{p_n} \big) + o(1).
 \end{equation}
Denote $\tilde{r}_l = \mathbbm{P}( \tilde{D}=l )= \lim_{n \to \infty} \tilde{n}_l/\tilde{n}$. 
Let $\tilde{D}_n$ denote the degree of a uniformly chosen vertex from $[\tilde{n}]$, independently of the graph $\mathrm{CM}_{\tilde{n}}(\Mtilde{\boldsymbol{d}})$.
Thus, \eqref{estimate:n-j-tilde} and \eqref{limit-n-tilde} imply that $\tilde{D}_n\xrightarrow{\sss\mathcal{L}} \tilde{D}$. The following lemma verifies the rest of the conditions for $\Mtilde{\boldsymbol{d}}$ in Assumption~\ref{assumption1}:
 \begin{lemma} \label{lem_percolation_condition} The statements below are true  $\mathbbm{P}_p$ almost surely{\rm :}
  \begin{enumerate}
   \item[{\rm(1)}] Under Assumption~\ref{assumption2}~\ref{assumption2-1} and for $r= 1,2,3$,
    \begin{equation}
     \frac{1}{\tilde{n}}\sum_{i \in [n]} \tilde{d}_i^r =  \frac{1}{\tilde{n}}\sum_{j \in [n]} j^r \tilde{n}_j \xrightarrow{n \to \infty}
     \mathbbm{E} [ \tilde{D}^r ].
    \end{equation}
   \item[{\rm(2)}]  Under Assumption~\ref{assumption2},
    \begin{equation}
      \tilde{\nu}_n = 1+\lambda n^{-1/3}+o(n^{-1/3}).
    \end{equation}
\end{enumerate}
  \end{lemma}
\begin{proof}
We will make use of \cite[Corollary 2.27]{JLR00}. Suppose $Z_1$, $Z_2$, ..., $Z_N$ are independent random variables with $Z_i$ taking values in $\Lambda_i$ and $f:\prod_{i=1}^N \Lambda_i \to \mathbbm{R}$ satisfies the following:
 If two vectors $z,z' \in \prod_{i=1}^N \Lambda_i$ differ only in the $i^{th}$ coordinate, then $| f(z)- f(z') | \leq c_i$ for some constant $c_i$.
Then, for any $t>0$, the random variable $X= f(Z_1, Z_2, \dots , Z_N)$ satisfies
\begin{equation} \label{janson_lemma_bound}
 \mathbbm{P} \Big( \big| X- \mathbbm{E}[X] \big| > t \Big) \leq 2 \exp \Big( -\frac{t^2}{2\sum_{i =1}^{N} c_i^2} \Big).
\end{equation}
 Now let $I_{ij}$ denote the indicator of the $j^{th}$ half-edge corresponding to vertex $i$ to be kept after Algorithm~\ref{algo:3}~(S1). Then $I_{ij} \sim \text{Ber} (\sqrt{p_n})$ independently for $j \in [d_i]$, $i \in [n]$. Let
 \begin{equation}
 \mathbf{I}:= (I_{ij})_{j \in [d_i], i \in [n]} \ \text{ and }\  f_1(\mathbf{I}):=\sum_{i\in [n]} \tilde{d_i}(\tilde{d}_i-1).
 \end{equation}Note that $f_1(\mathbf{I})=\sum_{i\in [\tilde{n}]}\tilde{d}_i(\tilde{d}_i-1)$ since the degree one vertices do not contribute to the sum. One can check that, by changing the status of one half-edge corresponding to vertex $k$, we can change $f_1(\cdot)$ by at most $2(d_{k}+1)$. Therefore, \eqref{janson_lemma_bound} yields
 \begin{equation}\mathbbm{P}_p \Big( \Big|\sum_{i\in [n]} \tilde{d_i}(\tilde{d}_i-1)- p_n \sum_{i\in [n]} d_i(d_i-1) \Big| >t \Big) \leq 2 \exp \bigg( -\frac{t^2}{8\sum_{i \in [n]} d_i (d_{i}+1)^2}\bigg).
 \end{equation}
 By setting $t= n^{1/2+ \varepsilon}$ for some suitably small $\varepsilon >0$, using the finite third moment conditions and the Borel-Cantelli lemma we conclude that  $\mathbbm{P}_p$ almost surely,
 \begin{equation}\sum_{i\in [n]} \tilde{d_i}(\tilde{d}_i-1)= p_n \sum_{i\in [n]} d_i(d_i-1) +O(n^{1/2+\varepsilon}), 
 \end{equation}and in particular,
  \begin{equation}\label{estimate:nu-n-num}\sum_{i\in [\tilde{n}]} \tilde{d_i}(\tilde{d}_i-1)=\sum_{i\in [n]} \tilde{d_i}(\tilde{d}_i-1)= p_n \sum_{i\in [n]} d_i(d_i-1) +o(n^{2/3}).
 \end{equation}
 Similarly, take $f_2(\mathbf{I})=\sum_{i\in [n]}\tilde{d}_i(\tilde{d}_i-1)(\tilde{d}_i-2)$ and note that changing the status of one bond changes $f_2(\cdot)$ by at most $[2(d_k+1)]^2$. Thus, \eqref{janson_lemma_bound} gives
 \begin{equation}
 \begin{split}  &\mathbbm{P}_p \Big( \Big| f_2(\mathbf{I})- p_n^{3/2} \sum_{i\in [n]} d_i(d_i-1)(d_i-2) \Big| >t \Big) \\
 &\hspace{2cm}\leq 2 \exp \bigg( -\frac{t^2}{32\sum_{i \in [n]} d_i (d_{i}+1)^4} \bigg)\\
 &\hspace{2cm}  \leq \exp \bigg( -\frac{t^2}{32d_{\max}(d_{\max}+1)\sum_{i \in [n]}  (d_{i}+1)^3} \bigg),
  \end{split}
 \end{equation}which implies that, $\mathbbm{P}_p$ almost surely,
 \begin{equation}\label{estimate-third-mom-perc}
  \sum_{i\in [\tilde{n}]}\tilde{d}_i(\tilde{d}_i-1)(\tilde{d}_i-2)=\sum_{i\in [n]}\tilde{d}_i(\tilde{d}_i-1)(\tilde{d}_i-2)=p_n^{3/2}\sum_{i\in [n]} d_i(d_i-1)(d_i-2)+o(n),
 \end{equation}since $d_{\max}^2\sum_{i\in [n]}(d_i+1)^3=o(n^{5/3})$.
Now, to prove Lemma~\ref{lem_percolation_condition}~{\rm (1)}, note that the case $r=1$ follows by simply observing that $\sum_{i\in \tilde{n}}\tilde{d}_i=\sum_{i\in [n]}d_i$. The cases $r=2,3$ follow from \eqref{estimate:nu-n-num} and \eqref{estimate-third-mom-perc}.
Finally, to see Lemma~\ref{lem_percolation_condition}~{\rm (2)}, note that
\begin{equation}
\begin{split}
  \tilde{\nu}_n  & = \frac{\sum_{i \in [\tilde{n}]}\tilde{d}_i(\tilde{d}_i-1)}{\sum_{i \in [\tilde{n}]}\tilde{d}_i}
 								  =\frac{p_n \sum_{i \in [n]} d_i \big( d_i -1 \big) + o \big( n^{2/3} \big)}{\sum_{i \in [n]}d_i} \\
 								 &= \frac{p_n \sum_{i \in [n]} d_i(d_i-1)}{\sum_{i \in [n]} d_i}+o(n^{-1/3})= 1+\frac{\lambda}{n^{1/3}}+o(n^{-1/3}),
 \end{split}
\end{equation} by \eqref{estimate:nu-n-num} and this completes the proof of Lemma~\ref{lem_percolation_condition}.
\end{proof}
We will denote by $\tilde{\mathscr{C}}_{\sss (j)}$, the $j^{th}$ largest component of $\mathrm{CM}_{\tilde{n}}(\Mtilde{\boldsymbol{d}})$. 
To conclude Theorem~\ref{thm_percolation} we also need to estimate the number of deleted vertices from each component. Recall from Remark~\ref{rem:alt-s3} that $\mathrm{CM}_n(\boldsymbol{d},p_n(\lambda))$ can be obtained from $\mathrm{CM}_{\tilde{n}}(\Mtilde{\boldsymbol{d}})$ by deleting relevant number of degree one vertices \emph{uniformly} at random. Let $v^d_1(\tilde{\mathscr{C}}_{\sss (j)})$ be the number of degree one vertices of $\tilde{\mathscr{C}}_{\sss (j)}$ that are deleted while creating $\mathrm{CM}_{n}(\boldsymbol{d},p_{n}(\lambda))$ from  $\mathrm{CM}_{\tilde{n}}(\Mtilde{\boldsymbol{d}})$. Since the vertices are to be chosen uniformly from all degree one vertices, the number of vertices to be deleted from $\tilde{\mathscr{C}}_{\sss (j)}$ is asymptotically the total number of degree one vertices in $\tilde{\mathscr{C}}_{\sss (j)}$ times the proportion of degree one  vertices to be deleted. Therefore,
\begin{equation} \label{degree_one_vertices}
\begin{split}
 v^d_1(\tilde{\mathscr{C}}_{\sss (j)}) &= \frac{n_+}{\tilde{n}_1}v_1(\tilde{\mathscr{C}}_{\sss (j)})+ o_{\sss\mathbbm{P}}(n^{2/3})= \frac{n_+}{\tilde{n}_1} \frac{\tilde{n}_1}{\sum_{k=0}^{\infty} k\tilde{n}_k} \big| \tilde{\mathscr{C}}_{\sss (j)}\big| + o_{\sss\mathbbm{P}}(n^{2/3})\\
 & = \frac{n_+}{\ell_n}\big| \tilde{\mathscr{C}}_{\sss (j)} \big|+ o_{\sss\mathbbm{P}}(n^{2/3})= \frac{\mathbbm{E}[D]\big(1-\sqrt{p}_n\big)}{\mathbbm{E}[D]}\big| \tilde{\mathscr{C}}_{\sss (j)} \big|+ o_{\sss\mathbbm{P}}(n^{2/3})\\
 & = \big(1-\sqrt{p}_n\big) \big| \tilde{\mathscr{C}}_{\sss (j)}\big| + o_{\sss\mathbbm{P}}(n^{2/3}),
 \end{split}
\end{equation} where the third equality follows from \eqref{eqn_vertices_of_degree_k-ord}. 
The proof of Theorem \ref{thm_percolation} is now complete by using the $\ell^{2}_{\shortarrow}$ convergence in Lemma~\ref{vertices_of_degree_k_ord}, \eqref{degree_one_vertices} and Remark~\ref{remark-perc}.

\section{Joint convergence at multiple locations in the critical window} \label{sec_multidimensional}
We will prove Theorem~\ref{thm_multiple_convergence} in this section. 
In Section~\ref{sec:perc-alt-cons}, we give a construction of the joint distribution of the percolated graphs for different percolation parameters that are coupled in a way described in Theorem~\ref{thm_multiple_convergence}.
In Section~\ref{sec:dynamic-construction}, we compare the process of percolated graphs with a different graph process that turns out to be easier to work with.
As discussed in Remark~\ref{rem:mult-coal-heuristics}, let the mass of a component be the number of open half-edges (re-scaled by $n^{2/3}$). 
The alternatively constructed graph process can be modified in such a way that the vector of masses evolves according to an \emph{exact} multiplicative coalescent as discussed in Section~\ref{sec:modified-C1}. 
Thus the joint convergence result at multiple locations of the scaling window can be deduced for the modified process using the Feller property of the multiplicative coalescent.
Further, the modified process remains \emph{close} to the dynamic construction. 
In Section~\ref{sec:open-he}, the vector of masses are shown to be asymptotically proportional to the component sizes and we combine all the above observations in Section~\ref{sec-mul-conv-thm-proof} to complete the proof of Theorem~\ref{thm_multiple_convergence}.
\subsection{Construction of the percolated graph process} 
\label{sec:perc-alt-cons}
We start by explaining a way to construct the graph process $(\mathrm{CM}_n(\boldsymbol{d},p_n(\lambda)))_{\lambda\in [\lambda_\star,\lambda^\star]}$, for any $-\infty<\lambda_\star<\lambda^\star<\infty$. 
Fix any $p_1<p_2<\dots<p_m$ and consider $(\mathrm{CM}_n(\bld{d},p_i))_{i\in [m]}$. 
Recall that each edge $e$ of $\CM$ has an independent uniform $[0,1]$ random variable $U_e$ associated to it and  $\mathrm{CM}_n(\bld{d},p_i)$ is obtained from $\CM$
 by keeping only those edges $e$ with $U_e\leq p_i$. This couples the graphs $(\mathrm{CM}_n(\bld{d},p_i))_{i\in [m]}$. 
 Moreover, under this coupling, $\mathrm{CM}_n(\bld{d},p_i)$ is distributed as the graph obtained from edge percolation on $\mathrm{CM}_n(\bld{d},p_{i+1})$ with probability $p_i/p_{i+1}$ for all $i<m$. 
 The following two lemmas are modifications of \cite[Lemmas~3.1,~3.2]{F07} that lead to the construction Algorithm~\ref{algo:cons-perc} below.
  For a graph $G$, let $\rE(G)$ denote the set of edges of $G$. For a sub-graph $G$ of $\CM$, let $\mathcal{H}(G)$ denote the set of half-edges that are part of some edge in $G$ and $\mathcal{H} = \mathcal{H}(\CM)$.
\begin{lemma}\label{lem:perc-cons-1} For $k_1\leq \dots\leq k_m$, conditionally on $\{|\rE(\mathrm{CM}_n(\bld{d},p_i))| = k_i:i\leq m\}$, the half-edges in $\mathrm{CM}_n(\bld{d},p_i)$ can be generated sequentially as follows: Let $k_0 = 0$, $\mathcal{H} (\mathrm{CM}_n(\bld{d},p_{0}))= \varnothing$. For each $i\leq m$, declare $\mathcal{H}(\mathrm{CM}_n(\bld{d},p_{i})) = \mathcal{H}(\mathrm{CM}_n(\bld{d},p_{i-1}))\cup \mathcal{H}_i$, where $\mathcal{H}_i$ is uniformly chosen among all the subsets of size $2k_i-2k_{i-1}$ of $\mathcal{H}\setminus \cup_{j<i}\mathcal{H}_i$.
\end{lemma} 
\begin{lemma}\label{lem:perc-cons-2} Let $d_k(i,i+1) $ be the number of half-edges attached to vertex $k$ in the graph $\mathrm{CM}_n(\bld{d},p_{i+1})$ that are not in $\mathrm{CM}_n(\bld{d},p_i)$. For any $i\geq 1$, conditionally on the event $\{\bld{d}(j,j+1) = \bld{d}_0(j,j+1):j\leq m\}$ and $\mathcal{H}(\mathrm{CM}_n(\bld{d},p_{i-1}))$, the perfect matching of $\mathcal{H}(\mathrm{CM}_n(\bld{d},p_{i}))\setminus \mathcal{H}(\mathrm{CM}_n(\bld{d},p_{i-1}))$ constituting the edges $\rE(\mathrm{CM}_n(\bld{d},p_{i})\setminus\mathrm{CM}_n(\bld{d},p_{i-1}))$ is a uniform perfect matching, where we have assumed that $p_0 = 0$.
\end{lemma}
\begin{algo}\label{algo:cons-perc} \normalfont Let $(U_i)_{i\geq 1}$ be a finite collection of i.i.d uniform $[0,1]$ random variables. Construct a collection of graphs $(G_n(\lambda))_{\lambda\in\R}$ using the following two steps:
\begin{itemize}
 \item[{\rm(S0)}]  Construct the process $\bld{E}_n = (E_n(\lambda))_{\lambda\in\R}$, where $E_n(\lambda) = \#\{i:U_i\leq p_n(\lambda)\}$.
 \item[{\rm(S1)}] Initially, $G_n(-\infty)$ is a graph only consisting of isolated vertices with no paired half-edges. At each time point $\lambda$ where $E_n(\lambda)$ has a jump, choose two unpaired half-edges uniformly at random and pair them. The graph $G_n(\lambda)$ is obtained by adding this edge to $G_n(\lambda-)$. 
\end{itemize} 
\end{algo} 
Algorithm~\ref{algo:cons-perc}~(S0) can be regarded as the birth of edges and Algorithm~\ref{algo:cons-perc}~(S1) ensures that the edges of the graph $G_n(\lambda)$ are obtained from a uniform perfect matching of the corresponding half-edges. 
Using Lemmas~\ref{lem:perc-cons-1}~and~\ref{lem:perc-cons-2}, the graph processes $(G_n(\lambda))_{\lambda\in\R}$ and $(\mathrm{CM}_n(\bld{d},p_n(\lambda)))_{\lambda\in \R}$ have the same finite-dimensional distributions. Therefore, for each fixed $n$, it follows that $(G_n(\lambda))_{\lambda\in\R}$ and $(\mathrm{CM}_n(\bld{d},p_n(\lambda)))_{\lambda\in \R}$ have the exact same distribution. 
We complete this section by adding proofs of Lemmas~\ref{lem:perc-cons-1},~and~\ref{lem:perc-cons-2} which are in the same spirit as the arguments of \cite[Lemmas~3.1,~3.2]{F07}.
\begin{proof}[Proof of Lemma~\ref{lem:perc-cons-1}] Assume that $k=2$ for the sake of simplicity. 
Observe that the total number of perfect matchings of $2k$ objects is given by $2k!/(k!2^k) = (2k-1)!!$.
Let $H_1$, $H_2$ be two disjoint subsets of $\mathcal{H}$ with $|H_1| = 2k_1$, $|H_2| = 2k_2-2k_1$.
Let $\mathcal{E}_1$ denote the event that a uniform perfect matching of all the half-edges contains also perfect matchings of the half-edges in $H_1$ and $H_2$. Then,
\begin{equation}\label{prob-E1}
\prob{\mathcal{E}_1} = \frac{(2k_1-1)!! (2k_2-2k_1-1)!! (\ell_n-2k_2-1)!!}{(\ell_n-1)!!}.
\end{equation}
Also, for percolation on any (random) graph, conditional on the set of edges of the graph and the fact that $k$ edges have been retained by percolation, the choice of the retained edges is uniformly distributed among all subsets of size $k$ of the set of edges. Let $\mathcal{E}_2$ denote the event that $|\mathcal{H}(\mathrm{CM}_n(\bld{d},p_{1}))| = 2k_1$, and  $|\mathcal{H}(\mathrm{CM}_n(\bld{d},p_{2}))| = 2k_2$. It follows that
\begin{equation}
\prob{\mathcal{H}(\mathrm{CM}_n(\bld{d},p_{2}))=H_1\cup H_2\mid \mathcal{E}_1,\mathcal{E}_2}= \frac{1}{\binom{\ell_n/2}{k_2}},\end{equation}
and
\begin{equation}
\prob{\mathcal{H}(\mathrm{CM}_n(\bld{d},p_{1}))=H_1\mid \mathcal{E}_1,\mathcal{E}_2,\mathcal{H}(\mathrm{CM}_n(\bld{d},p_{2}))=H_1\cup H_2} = \frac{1}{\binom{k_2}{k_1}}.
\end{equation}
 Thus, conditional on the event $\mathcal{E}_2$, the probability that $\mathcal{H}(\mathrm{CM}_n(\bld{d},p_{1})) = H_1$ and $\mathcal{H}(\mathrm{CM}_n(\bld{d},p_{2}))\setminus \mathcal{H}(\mathrm{CM}_n(\bld{d},p_{1}))=H_2$
 is given by 
 \begin{equation}\label{perfect-matching-1}
 \begin{split}
 \frac{(2k_1-1)!! (2k_2-2k_1-1)!! (\ell_n-2k_2-1)!!}{(\ell_n-1)!!} \frac{1}{\binom{\ell_n/2}{k_2}\binom{k_2}{k_1}} = \frac{1}{\binom{\ell_n}{2k_1}\binom{\ell_n-2k_1}{2k_2-2k_1}},
 \end{split}
 \end{equation}which does not depend on $H_1$ or $H_2$, and the proof follows.
\end{proof}
\begin{proof}[Proof of Lemma~\ref{lem:perc-cons-2}]
 Fix two disjoint subsets $H_1$, $H_2$ of $\mathcal{H}$ such that $|H_1| = 2k_1$, $|H_2| = 2k_2-2k_1$.
 As in the proof of Lemma~\ref{lem:perc-cons-1}, let $\mathcal{E}_2$ denote the event that $|\mathcal{H}(\mathrm{CM}_n(\bld{d},p_{1}))| = 2k_1$, and  $|\mathcal{H}(\mathrm{CM}_n(\bld{d},p_{2}))| = 2k_2$.
An identical argument as the proof of \eqref{perfect-matching-1} now gives, conditionally on $\mathcal{E}_2$, the probability that $\mathcal{H}(\mathrm{CM}_n(\bld{d},p_1)) = H_1$, $\mathcal{H}(\mathrm{CM}_n(\bld{d},p_2))\setminus \mathcal{H}(\mathrm{CM}_n(\bld{d},p_1)) = H_2$, and given perfect matchings on $\mathcal{H}(\mathrm{CM}_n(\bld{d},p_{1}))$, $\mathcal{H}(\mathrm{CM}_n(\bld{d},p_2))\setminus \mathcal{H}(\mathrm{CM}_n(\bld{d},p_1))$ have been observed, is given by 
\begin{equation}\label{perfect-matching-2}
\begin{split}
\frac{1}{\binom{\ell_n/2}{k_2}\binom{k_2}{k_1}}\frac{(\ell_n-2k_2-1)!!}{(\ell_n-1)!!}.
\end{split}
\end{equation} 
Let $\rD(H)$ denote the degree sequence induced by the set of half-edges $H$, and  $S$ denote the collection of \emph{disjoint} pairs $(H_1,H_2)$ such that $|H_1| = 2k_1$, $|H_2| = 2k_2-2k_1$, $\rD(H_1) = \bld{d}_0(0,1)$, and $\rD(H_2) = \bld{d}_0(1,2)$.
Then, conditionally on $\mathcal{E}_2$, the probability that $\bld{d}(0,1) = \bld{d}_0(0,1)$,  $\bld{d}(1,2) = \bld{d}_0(1,2)$, and given particular perfect matchings have been observed on $\mathcal{H}(\mathrm{CM}_n(\bld{d},p_1))$ and $\mathcal{H}(\mathrm{CM}_n(\bld{d},p_2))\setminus \mathcal{H}(\mathrm{CM}_n(\bld{d},p_1))$, is  
\begin{equation}\label{perfect-matching-3}
 \sum_{(H_1,H_2)\in S} \frac{1}{\binom{\ell_n/2}{k_2}\binom{k_2}{k_1}}\frac{(\ell_n-2k_2-1)!!}{(\ell_n-1)!!}  = \frac{|S|}{\binom{\ell_n/2}{k_2}\binom{k_2}{k_1}}\frac{(\ell_n-2k_2-1)!!}{(\ell_n-1)!!}.
\end{equation}
Moreover, by Lemma~\ref{lem:perc-cons-1}, the probability that $\bld{d}(0,1) = \bld{d}_0(0,1)$,  $\bld{d}(1,2) = \bld{d}_0(1,2)$, conditionally on $\mathcal{E}_2$, is given by 
  \begin{equation}\label{perfect-matching-4}
   \frac{|S|}{\binom{\ell_n}{2k_1}\binom{\ell_n-2k_1}{2k_2-2k_1}}. 
  \end{equation}
Now,  \eqref{perfect-matching-3} and \eqref{perfect-matching-4} together yield that the probability that two particular perfect matchings are observed on $\mathcal{H}(\mathrm{CM}_n(\bld{d},p_1))$ and $\mathcal{H}(\mathrm{CM}_n(\bld{d},p_2))\setminus \mathcal{H}(\mathrm{CM}_n(\bld{d},p_1))$, conditional on $\bld{d}(0,1) = \bld{d}_0(0,1)$,  $\bld{d}(1,2) = \bld{d}_0(1,2)$ is given by 
\begin{equation}
\frac{1}{\binom{\ell_n/2}{k_2}\binom{k_2}{k_1}}\frac{(\ell_n-2k_2-1)!!}{(\ell_n-1)!!}\binom{\ell_n}{2k_1}\binom{\ell_n-2k_1}{2k_2-2k_1} = \frac{1}{(2k_1-1)!!(2k_2-2k_1-1)!!},
\end{equation}
 and the proof is complete.
\end{proof}

\subsection{The dynamic construction}\label{sec:dynamic-construction}
Let us now describe a dynamic construction of $\CM$ that turns out to be easier to work with. This dynamic construction was introduced in \cite{BBSX14} to study the metric-space limits of the large components of the percolated configuration model. 
It will be shown that the graphs generated by this dynamic construction at a suitable range of time \emph{approximates} the  process $(\mathrm{CM}_n(\boldsymbol{d},p_n(\lambda)))_{\lambda\in\R}$.  
\begin{algo} \label{algo-dyn-cons} \normalfont At time $t=0$, assume that there are $d_i$ \emph{open} half-edges associated with vertex $i$, for all $i\in [n]$. Associate i.i.d unit rate exponential clocks to each of the open half-edges. Each time an exponential clock rings, the corresponding half-edge selects another open half-edge uniformly at random and gets paired to it. The two paired half-edges are declared to be closed and the associated exponential clocks are removed. The process continues until the open half-edges are exhausted.
\end{algo}
Let $\mathcal{G}_n(t)$ denote the graph generated upto time $t$. Notice that $\mathcal{G}_n(\infty)$ is distributed as $\CM$ since each half-edge chooses to pair with another uniformly chosen open half-edge. 
Denote the total number of open-half-edges remaining at time $t$ while implementing Algorithm~\ref{algo-dyn-cons} by $s_1(t)$. 
The graph process, given by Algorithm~\ref{algo-dyn-cons}, can also be constructed as follows:
\begin{algo}\label{algo:dyn-cons-alt} \normalfont Let $\Xi_n$ be an inhomogeneous Poisson process with rate $s_1(t)$ at time $t$. Let $e_1<e_2<\dots$ be the event times of $\Xi_n$.
\begin{itemize}
\item[{\rm(S1)}] At each event time, choose two unpaired half-edges uniformly at random and pair them. The graph $\mathcal{G}_n(t)$ is obtained by adding this edge to $\mathcal{G}_n(t-)$. 
\end{itemize} 
\end{algo} 
Notice the similarity between Algorithm~\ref{algo:cons-perc}~(S1) and Algorithm~\ref{algo:dyn-cons-alt}~(S1). 
Now, the idea is to compare the number of half-edges that have been paired by Algorithms~\ref{algo:cons-perc}~and~\ref{algo:dyn-cons-alt}.
For that, we need the following lemma that describes the evolution of the count of the total number of open half-edges in Algorithm~\ref{algo:dyn-cons-alt}:
\begin{lemma}[{\cite[Lemma 8.2]{BBSX14}}]\label{lem:total-open-he}  Let $s_1(t)$ denote the total number of open half-edges at time $t$. Suppose that 
Assumption \ref{assumption2} holds. Then, for any $T>0$ and some $1/3<\gamma<1/2$, 
\begin{equation}\label{eqn:s-1-he}
 \sup_{t\leq T}\Big|\frac{1}{\ell_n}s_1(t)- \e^{-2t}\Big|= \oP(n^{-\gamma}).
\end{equation}
\end{lemma}
 Notice that the proof of \cite[Lemma 8.2]{BBSX14} is stated only under some more stringent assumptions, however the identical argument can be carried out under Assumption~\ref{assumption2}.
The next proposition ensures that the graphs generated by percolation in Algorithm~\ref{algo:cons-perc} and the dynamic construction in Algorithm~\ref{algo-dyn-cons} are uniformly close in the critical window.
Define
\begin{equation}\label{defn:t-n-lambda}
t_n(\lambda)=\frac{1}{2}\log\bigg(\frac{\nu_n}{\nu_n-1}\bigg)+\frac{1}{2(\nu_n-1)}\frac{\lambda}{n^{1/3}}.
\end{equation}
\begin{proposition}\label{prop:coupling-whp} Fix $-\infty<\lambda_\star<\lambda^\star<\infty$. There exists a coupling such that with high probability
\begin{equation}
 \mathcal{G}_n(t_n(\lambda)-\varepsilon_n)\subset \mathrm{CM}_n(\bld{d},p_n(\lambda)) \subset\mathcal{G}_n(t_n(\lambda)+\varepsilon_n),\quad \forall \lambda \in [\lambda_\star,\lambda^\star]
\end{equation}where $\varepsilon_{n}=cn^{-\gamma_0}$, for some $1/3<\gamma_0<1/2$ and the constant $c$ does not depend on $\lambda$.
\end{proposition}
\begin{proof}
Notice the similarity between Algorithm~\ref{algo:cons-perc}~(S1) and Algorithm~\ref{algo:dyn-cons-alt}~(S1). 
Let $\#\mathrm{E}(G)$ denote the number of edges in a graph $G$.
Suppose that we can show, as $n\to\infty$,
\begin{equation}\label{eq:coup-reduc}
 \PR\big(\#\rE(\mathcal{G}_n(t_n(\lambda)-\varepsilon_n))\leq \#\rE(\mathrm{CM}_n(\bld{d},p_n(\lambda))) \leq \#\rE(\mathcal{G}_n(t_n(\lambda)+\varepsilon_n)), \forall \lambda \in [\lambda_\star,\lambda^\star]\big) \to  1.
\end{equation} 
On the event $\{\#\rE(\mathrm{CM}_n(\bld{d},p_n(\lambda))) \leq \#\rE(\mathcal{G}_n(t_n(\lambda)+\varepsilon_n)), \forall \lambda \in [\lambda_\star,\lambda^\star]\}$, the choice of the uniform pair of half-edges at the $k^{th}$ pairing in Algorithm~\ref{algo:cons-perc}~(S1)  can be taken to be exactly same as the $k^{th}$ pairing in  Algorithm~\ref{algo:dyn-cons-alt}~(S1).
 Under the above coupling $\mathrm{CM}_n(\bld{d},p_n(\lambda_\star)) \subset\mathcal{G}_n(t_n(\lambda_\star)+\varepsilon_n)$.
 Moreover, since $\#\rE(\mathrm{CM}_n(\bld{d},p_n(\lambda)))$ is dominated by  $\#\rE(\mathcal{G}_n(t_n(\lambda)+\varepsilon_n))$, uniformly over  $\lambda \in [\lambda_\star,\lambda^\star]$,  
the above coupling  also yields that $ \mathrm{CM}_n(\bld{d},p_n(\lambda)) \subset\mathcal{G}_n(t_n(\lambda)+\varepsilon_n)$ for all $\lambda\in [\lambda_\star,\lambda^\star]$. Further, on the event $\{ \#\rE(\mathcal{G}_n(t_n(\lambda)-\varepsilon_n))\leq\#\rE(\mathrm{CM}_n(\bld{d},p_n(\lambda))) , \forall \lambda \in [\lambda_\star,\lambda^\star]\}$, under the same coupling, $\mathcal{G}_n(t_n(\lambda)-\varepsilon_n)\subset\mathrm{CM}_n(\bld{d},p_n(\lambda)) $ for all $\lambda\in [\lambda_\star,\lambda^\star]$. 
Thus, it remains to show \eqref{eq:coup-reduc}. An application of Lemma~\ref{lem:total-open-he} along with \eqref{defn:t-n-lambda} yields, for some $1/3<\gamma_0<\gamma<1/2$, with high probability,
\begin{equation}\label{edges-dyn-cons}
\bigg| \#\mathrm{E}(\mathcal{G}_n(t_n(\lambda))) - \bigg(\frac{\ell_n}{2\nu_n} +\frac{\lambda\ell_n}{2\nu_nn^{1/3}} +\frac{n\varepsilon_n(\nu_n-1)}{\nu_n}\bigg)\bigg| \leq n^{1-\gamma}, \quad \lambda \in [\lambda_\star,\lambda^\star].
\end{equation}
Notice that the total number of half-edges in $\mathrm{CM}_n(\bld{d},p_n(\lambda))$ follows a binomial distribution with parameters $\ell_n/2$ and $p_n(\lambda)$. Thus, with high probability,
\begin{equation}\label{edges-perc}
 \bigg|\#\mathrm{E}(\mathrm{CM}_n(\bld{d},p_n(\lambda)))-\bigg(\frac{\ell_n}{2\nu_n} +\frac{\lambda\ell_n}{2\nu_nn^{1/3}}\bigg)\bigg|\leq n^{1-\gamma}, \quad \lambda \in [\lambda_\star,\lambda^\star].
\end{equation}The fact that the error can be chosen to be uniform over $\lambda\in [\lambda_\star,\lambda^\star]$ follows from the DKW inequality \cite{M90}. Thus,  \eqref{edges-dyn-cons} and \eqref{edges-perc} together show that, with high probability,
\begin{equation}
 \#\rE(\mathrm{CM}_n(\bld{d},p_n(\lambda))) \leq \#\rE(\mathcal{G}_n(t_n(\lambda)+\varepsilon_n)), \quad \forall \lambda \in [\lambda_\star,\lambda^\star].
\end{equation}The other part follow similarly and the proof is now complete.
\end{proof}
\begin{remark}\label{rem:modified-prop-coup} \normalfont Notice that the proof of Proposition~\ref{prop:coupling-whp} can be directly modified to show that there exists a coupling such that, with high probability,
\begin{equation}
 \mathrm{CM}_n(\bld{d},p_n(\lambda)-\varepsilon_n)\subset \mathcal{G}_n(t_n(\lambda))\subset \mathrm{CM}_n(\bld{d},p_n(\lambda)+\varepsilon_n),\quad \forall \lambda \in [\lambda_\star,\lambda^\star]
\end{equation}where $\varepsilon_{n}=cn^{-\gamma_0}$, for some $1/3<\gamma_0<1/2$ and the constant $c$ does not depend on~$\lambda$. Therefore, the scaling limits of different functionals like re-scaled component-sizes, surplus edges for $\mathcal{G}_n(t_n(\lambda))$ and $\mathrm{CM}_n(\bld{d},p_n(\lambda))$ are the same. 
\end{remark}
\subsection{The modified process}\label{sec:modified-C1}
From here onward, we often augment $\lambda$ to a predefined notation to emphasize the dependence on~$\lambda$. 
We write $\mathscr{C}_{\sss (i)}(\lambda)$  for the $i^{th}$ largest component of $\mathcal{G}_n(t_n(\lambda))$ and define 
\begin{equation}
\mathcal{O}_i(\lambda)=\# \text{ open half-edges in }\mathscr{C}_{\sss (i)}(\lambda).
\end{equation}
Think of $\mathcal{O}_i(\lambda)$ as the \emph{mass} of the component $\mathscr{C}_{\sss (i)}(\lambda)$. 
Let $\mathbf{C}_n(\lambda) = (n^{-2/3}|\mathscr{C}_{\sss (i)}(\lambda)|)_{i\geq 1}$, and $\mathbf{O}_n(\lambda) = (n^{-2/3}\mathcal{O}_i(\lambda))_{i \geq 1}$.
 Let $\ell_n^o(\lambda) = \sum_{i\geq 1}\mathcal{O}_i(\lambda)$. By Lemma~\ref{lem:total-open-he} and \eqref{defn:t-n-lambda}, $\ell_n^o(\lambda) \approx n\mu(\nu-1)/\nu$. Now, observe that, during the evolution of the graph process generated  by Algorithm~\ref{algo-dyn-cons}, between time $[t_n(\lambda),t_n(\lambda+\dif \lambda)]$, the $i^{th}$ and $j^{th}$ ($i> j$) largest components, merge at rate 
 \begin{equation}\label{rate:function}
2\mathcal{O}_{i}(\lambda) \mathcal{O}_{j}(\lambda)\times\frac{1}{\ell_n^o(\lambda)-1}\times \frac{1}{2(\nu_n-1)n^{1/3}}\approx \frac{\nu}{\mu(\nu-1)^2} \big(n^{-2/3}\mathcal{O}_{i}(\lambda)\big)\big(n^{-2/3}\mathcal{O}_{j}(\lambda)\big),
\end{equation}and creates a component with open half-edges $\mathcal{O}_{i}(\lambda)+\mathcal{O}_{j}(\lambda)-2$.
Thus $(\mathbf{O}_n(\lambda))_{\lambda\in\R}$ does \emph{not} evolve as a multiplicative coalescent, but it is close. 
The fact that two half-edges are killed after pairing, makes the masses (the number of open half-edges) of the components  and the system to deplete. 
If there were no such depletion of mass, then the vector of open half-edges would in fact  merge as multiplicative coalescent. 
Let us formalize this  idea below:
\begin{algo}\label{algo:modify-dyn-cons} \normalfont Initialize $\bar{\mathcal{G}}_n(t_n(\lambda_\star)) = \mathcal{G}_n(t_n(\lambda_\star))$.  Let $\mathscr{O}$ denote the set of open half-edges in the graph $\mathcal{G}_n(t_n(\lambda_\star))$, $\bar{s}_1 = |\mathscr{O}|$ and $\bar{\Xi}_n$ denote a Poisson process with rate $\bar{s}_1$. At each event time of the Poisson process $\bar{\Xi}_n$, select two half-edges from $\mathscr{O}$ and create an edge between the corresponding vertices. However, the selected half-edges are kept alive, so that they can be selected again.
\end{algo} 
\begin{remark}\label{rem:modify-AMC}\normalfont The only difference between Algorithm~\ref{algo:dyn-cons-alt} and Algorithm~\ref{algo:modify-dyn-cons}, is that the \emph{paired} half-edges are not discarded and thus more edges are created by Algorithm~\ref{algo:modify-dyn-cons}. Thus, there is a natural coupling between the graphs generated by Algorithms~\ref{algo:dyn-cons-alt}~and~\ref{algo:modify-dyn-cons} such that $\mathcal{G}_n(t_n(\lambda))\subset \bar{\mathcal{G}}_n(t_n(\lambda))$ for all $\lambda\in [\lambda_\star,\lambda^\star]$, with probability one. In the subsequent part of this section, we always work under this coupling. The extra edges that are created by Algorithm~\ref{algo:modify-dyn-cons} will be called \emph{bad} edges.
\end{remark}
\begin{remark} \label{rem:MC-exact-limit}\normalfont In the subsequent part of this paper, we shall augment a predefined notation with a bar to denote the corresponding quantity for $\bar{\mathcal{G}}_n(t_n(\lambda))$. 
 Denote $\beta_n = (\bar{s}_1(\nu_n-1)n^{1/3})^{1/2}$ and $\bar{\mathbf{O}}_n'(\lambda)$ denote the vector $\ord((\beta_n^{-1}\bar{\mathcal{O}}_i(\lambda))_{i\geq 1})$.  
 By the description in Algorithm~\ref{algo:modify-dyn-cons}, $(\bar{\mathbf{O}}_n'(\lambda))_{\lambda\geq \lambda_\star}$ evolves as a standard multiplicative coalescent.
Further, note that there exists a constant $c>0$ such that $\beta_n = cn^{2/3}(1+\oP(1))$ which enables us to deduce the scaling limit results for $(\bar{\mathbf{O}}_n(\lambda))_{\lambda\geq \lambda_\star}$ from $(\bar{\mathbf{O}}_n'(\lambda))_{\lambda\geq \lambda_\star}$.
\end{remark}

\subsubsection*{Multiplicative coalescent with mass and weight}
The Feller property of the multiplicative coalescent \cite[Proposition 5]{A97} ensures the joint convergence of the number of open half-edges in each component of $\bar{\mathcal{G}}_n(t_n(\lambda))$ at multiple values of $\lambda$ as we shall see below. 
To deduce the scaling limits involving the components sizes let us consider a dynamic process that is further augmented by a certain weight. 
Initially, the system consists of particles (possibly infinitely many) where particle $i$ has mass $x_i$, and weight $z_i$. 
Let $(X_i(t),Z_i(t))_{i\geq 1}$ denote the vector of masses, and weights at time $t$. 
The dynamics of the system is described as follows:
\begin{itemize}
\item[]   At time~$t$, particles $i$ and $j$ coalesce at rate $X_i(t)X_j(t)$ and create a particle with mass $X_i(t)+X_j(t)$, and weight $Z_i(t)+Z_j(t)$.
\end{itemize}
Denote by  $\mathrm{MC}_2(\mathbf{x},\mathbf{z},t)$ the vector $(X_i(t), Z_i(t))_{i\geq 1}$ with initial mass $\mathbf{x}$, and weight $\mathbf{z}$. 
We shall need the following theorem:
\begin{theorem}\label{thm:AMC-2D}
Suppose that $(\mathbf{x}_n,\mathbf{z}_n) \to (\mathbf{x},\mathbf{x})$ in $(\ell^2_{\shortarrow})^2$. Then, for any $t\geq 0$
\begin{equation}
\mathrm{MC}_2(\mathbf{x}_n,\mathbf{z}_n,t)\dto \mathrm{MC}_2(\mathbf{x},\mathbf{x},t).
\end{equation}
\end{theorem}
\begin{proof}
For $\mathbf{x}_n = (x_i^n)_{i\geq 1}$ and $\mathbf{z}_n = (z_i^n)_{i\geq 1}$,  let $\mathbf{w}_n^+ = \mathrm{ord}(x_i^n\vee z_i^n)$, $\mathbf{w}_n^-=\mathrm{ord}(x_i^n\wedge z_i^n)$, where $\mathrm{ord}$ denotes the decreasing ordering of the elements. 
Notice that $\mathbf{w}_n^+ \to \mathbf{x}$, and $\mathbf{w}_n^- \to \mathbf{x}$ in $\ell^2_{\shortarrow}$.
Using the Feller property of the multiplicative coalescent \cite[Proposition 5]{A97}, it follows that
\begin{equation}\label{limit-ub-lb}
 \mathrm{MC}_2(\mathbf{w}_n^+,\mathbf{w}_n^+,t)\dto  \mathrm{MC}_2(\mathbf{x},\mathbf{x},t), \quad \text{and} \quad \mathrm{MC}_2(\mathbf{w}_n^-,\mathbf{w}_n^-,t)\dto  \mathrm{MC}_2(\mathbf{x},\mathbf{x},t),
\end{equation}with respect to the $(\ell^2_{\shortarrow})^2$ topology. 
Suppose that $\mathrm{MC}_2(\mathbf{w}_n^+,\mathbf{w}_n^+,t)$ and $\mathrm{MC}_2(\mathbf{w}_n^-,\mathbf{w}_n^-,t)$ are coupled through the subgraph coupling (see \cite[Page 838]{A97}). 
For $(\mathbf{x},\mathbf{z})\in (\ell_{\shortarrow}^2)^2$, denote $\|(\mathbf{x},\mathbf{z})\|_{\sss 22} = (\sum_{i\geq 1}x_i^2)^{1/2}+(\sum_{i\geq 1}z_i^2)^{1/2}$.
Under the subgraph coupling, \eqref{limit-ub-lb} yields
\begin{equation}
\|\mathrm{MC}_2(\mathbf{w}_n^+,\mathbf{w}_n^+,t)\|_{\sss 22}^2-\|\mathrm{MC}_2(\mathbf{w}_n^-,\mathbf{w}_n^-,t)\|_{\sss 22}^2 \pto 0.
\end{equation}
Moreover,
\begin{equation}
 \|\mathrm{MC}_2(\mathbf{w}_n^-,\mathbf{w}_n^-,t)\|_{\sss 22}^2 \leq \|\mathrm{MC}_2(\mathbf{x}_n,\mathbf{z}_n,t)\|_{\sss 22}^2 \leq \|\mathrm{MC}_2(\mathbf{w}_n^+,\mathbf{w}_n^+,t)\|_{\sss 22}^2.
\end{equation}
Hence, using \cite[Corollary~18~(a)]{A97}, under the subgraph coupling,
\begin{equation}
 \|\mathrm{MC}_2(\mathbf{w}_n^+,\mathbf{w}_n^+,t) - \mathrm{MC}_2(\mathbf{x}_n,\mathbf{z}_n,t)\|_{\sss 22}^2\leq \|\mathrm{MC}_2(\mathbf{w}_n^+,\mathbf{w}_n^+,t)\|_{\sss 22}^2 - \|\mathrm{MC}_2(\mathbf{x}_n,\mathbf{z}_n,t)\|_{\sss 22}^2 \pto 0,
\end{equation}and the proof follows.
\end{proof}
\subsection{Asymptotics for the open half-edges}\label{sec:open-he}
In this section, we show that the open half-edges in the components of $\mathcal{G}_n(t_n(\lambda))$ are \emph{approximately} proportional to the component sizes. 
This will enable us to apply Theorem~\ref{thm:AMC-2D} for deducing the scaling limits of the required quantities for the graph $\bar{\mathcal{G}}_n(t_n(\lambda))$.
\begin{lemma}\label{thm:open-comp}
 There exists a constant $\kappa > 0$ such that, for any $\lambda\in\R$ and $i\geq 1$,
 \begin{equation}\label{open-he}
  \mathcal{O}_i(\lambda)= \kappa |\mathscr{C}_{\sss (i)}(\lambda)|+o_{\sss \PR}(b_n).
 \end{equation}Further, $(\mathbf{O}_n(\lambda))_{n\geq 1}$ is tight in  $\ell^2_{\shortarrow}$ and consequently $n^{-4/3}\sum_{i\geq 1} (\mathcal{O}_i(\lambda)- \kappa |\mathscr{C}_{\sss (i)}(\lambda)|)^2 \pto 0$. 
\end{lemma}
\begin{proof} Let $(d_k^\lambda)_{k\in [n]}$ denote the degree sequence of $\mathrm{CM}_n(\bld{d},p_n(\lambda))$ and define
\begin{equation}
\mathcal{O}_i^p(\lambda) = \sum_{k\in \mathscr{C}_{\sss (i)}^p(\lambda)}(d_k-d_k^\lambda) = \sum_{k\in \mathscr{C}_{\sss (i)}^p(\lambda)}d_k-2(|\mathscr{C}_{\sss (i)}^p(\lambda)|-1+\mathrm{SP}(\mathscr{C}_{\sss (i)}^p(\lambda))).
\end{equation}
 Using Remark~\ref{rem:modified-prop-coup} and the fact that the surplus edges in the large components is tight, it is enough to prove the lemma by replacing $\mathcal{O}_i(\lambda)$ by $\mathcal{O}_i^p(\lambda)$ and $\mathscr{C}_{\sss (i)}(\lambda)$ by $\mathscr{C}_{\sss (i)}^p(\lambda)$.
For a component $\tilde{\mathscr{C}}$ of $\mathrm{CM}_{\tilde{n}}(\Mtilde{\boldsymbol{d}})$, the corresponding component $\tilde{\mathscr{C}}^p$ in the percolated graph is obtained by cleaning up $R(\tilde{\mathscr{C}})$ red degree-one vertices, see Algorithm~\ref{algo:3}. 
Thus, the number of open half-edges in $\tilde{\mathscr{C}}^p$ is given by 
\begin{equation}\label{relation-deg-def}
 \sum_{k\in \tilde{\mathscr{C}}\cap [n]}d_k-\sum_{k\in \tilde{\mathscr{C}}\cap [n]}\tilde{d}_k+R(\tilde{\mathscr{C}}).
\end{equation} 
Now, all the three terms appearing in the right hand side of \eqref{relation-deg-def} can be estimated using Lemma~\ref{lem_general}. 
Indeed, we can consider weights $w_{i1} =d_i$, $w_{i2}=\tilde{d}_i$, and $w_{i3}=$ the number of red neighbors of vertex $i$ in $\mathrm{CM}_{\tilde{n}}(\Mtilde{\bld{d}})$. The conditions in~\eqref{eq:conditions-size-biased} are satisfied by Lemma~\ref{lem_percolation_condition}, and observing that 
\begin{equation}
 \max\{\max_iw_{i1}, \max_iw_{i2}, \max_{i}w_{i3}\}\leq d_{\max} = o(n^{1/3}).
\end{equation}
Note that, using an argument identical to Lemma~\ref{lem_percolation_condition}, $(1/n) \sum_{i\in [\tilde{n}]}w_{ik}\tilde{d}_i$  converges $\PR_p$ almost surely, for all $k=1,2,3$. 
Now, \eqref{open-he} is a consequence of Lemma~\ref{lem:large-com-explored-early}.
Denote 
\begin{equation}
D_i = \sum_{k\in \tilde{\mathscr{C}}_{\sss (i)}\cap [n]}d_k, \quad\tilde{D}_i = \sum_{k\in \tilde{\mathscr{C}}_{\sss (i)}\cap [n]}\tilde{d}_k,\quad \mathbf{D}_n = \mathrm{ord}((D_i)_{i\geq 1}),\quad \text{and}\quad \tilde{\mathbf{D}}_n = \mathrm{ord}((\tilde{D}_i)_{i\geq 1}).
\end{equation} 
Using \eqref{eqn_tightness}, $(\tilde{\mathbf{D}}_n)_{n\geq 1}$ is tight in $\ell^2_{\shortarrow}$. Further $w_{i3}\leq d_i$ for all $i$. Thus, for the $\ell^2_{\shortarrow}$ tightness of $(\mathbf{O}_n(\lambda))_{n\geq 1}$, it is enough to show the $\ell^2_{\shortarrow}$ tightness of $(\mathbf{D}_n)_{n\geq 1}$.
Denote the conditional probability, conditioned on the uniform perfect matching in Algorithm~\ref{algo:3}~(S2), by $\tilde{\PR}(\cdot)$. 
Notice that, since Algorithm~\ref{algo:3}~(S1), and~(S2) are carried out independently, $\tilde{D}_i \sim \mathrm{Bin}(D_i,\sqrt{p_n})$ under $\tilde{\PR}$. 
Using standard concentration inequalities \cite[(2.9)]{JLR00}, it follows that 
\begin{equation}
 \tilde{\PR}(\tilde{D}_i< D_i\sqrt{p_n}(1-\sqrt{p_n}))\leq 2\e^{-D_ip_n^{3/2}/3},
\end{equation}and thus for $\mathcal{I} = \{k:D_k>n^{\varepsilon}\}$, the union bound yields
\begin{equation}\label{domination-tilde-orginal-degree}
  \PR(\exists i\in \mathcal{I}:D_i > a\tilde{D}_i) \to 0,
\end{equation}for some constant $a>0$. Let $\mathcal{E}_n$ denote the corresponding event in \eqref{domination-tilde-orginal-degree}. Thus, for any $\eta >0$, 
\begin{equation}
 \PR\bigg(n^{-4/3}\sum_{k>K,k\in \mathcal{I}}D_k^2>\eta\bigg) \leq \PR\bigg(n^{-4/3}\sum_{k>K}\tilde{D}_k^2>\frac{\eta}{a}\bigg)+ \PR(\mathcal{E}_n) \to 0,
\end{equation}if we first take first take limit as $n\to\infty$, and then $K\to\infty$, and use the $\ell^2_{\shortarrow}$ tightness of $(\tilde{\mathbf{D}}_n)_{n\geq 1}$. 
Further, $\sum_{k\notin \mathcal{I}}D_k^{2}\leq n^{1+2\varepsilon}=o(n^{4/3})$, if $\varepsilon<1/6$. This completes the proof of the $\ell^2_{\shortarrow}$ tightness of $(\mathbf{D}_n)_{n\geq 1}$ and consequently that of $(\mathbf{O}_n(\lambda))_{n\geq 1}$.
\end{proof}

\subsection{Proof of Theorem~\ref{thm_multiple_convergence}}
 \label{sec-mul-conv-thm-proof}
We will consider the case $k=2$ only, since the case for general $k$ can be proved inductively. Fix $-\infty<\lambda_0<\lambda_1<\infty$.
Suppose that the modified Algorithm~\ref{algo:modify-dyn-cons} starts at time $\lambda_{\star} = \lambda_0$. 
By Lemma~\ref{thm:open-comp} and Theorem~\ref{thm_percolation}, $(\mathbf{O}_n(\lambda_0),\kappa\mathbf{C}_n(\lambda_0))$ converges in distribution to $\kappa\sqrt{\nu}(\tilde{\bld{\gamma}}^{\lambda_0},\tilde{\bld{\gamma}}^{\lambda_0})$. Now, from Remark~\ref{rem:MC-exact-limit}, an application of Theorem~\ref{thm:AMC-2D} gives
\begin{equation}\label{joint-conv-2}
 (\mathbf{C}_n(\lambda_0),\bar{\mathbf{C}}_n(\lambda_1)) \dto \sqrt{\nu}(\tilde{\bld{\gamma}}^{\lambda_0},\tilde{\bld{\gamma}}^{\lambda_1}).
\end{equation}
The fact that the limiting distribution corresponding to $\bar{\mathbf{C}}_n(\lambda_1)$ is equal to $\sqrt{\nu}\tilde{\bld{\gamma}}^{\lambda_1}$ follows from the Feller property of multiplicative coalescent, \cite[Theorem 2]{AL98}, and Theorem~\ref{thm:AMC-2D}. For $\mathbf{x},\mathbf{y}\in \ell^2_{\shortarrow}$, denote $\mathbf{x}\preceq\mathbf{y}$ if $\mathbf{x}$ is the vector in decreasing order of elements $\{y_{ij}:i,j\geq 1\}$ such that $\sum_{j}y_{ij}\leq y_i$ for all $i\geq 1$. 
Thus if $\mathbf{y}$ is obtained by \emph{coalescing} elements of $\mathbf{x}$, then $\mathbf{x}\preceq\mathbf{y}$. 
Under the coupling in Remark~\ref{rem:modify-AMC}, it follows that $\mathbf{C}_{n}(\lambda)\preceq \bar{\mathbf{C}}_{n}(\lambda)$ almost surely, for each $\lambda\geq \lambda_0$. 
Using \cite[Corollary~18~(a)]{A97}, it follows that
\begin{equation}\label{bound-norm}
\|\bar{\mathbf{C}}_n(\lambda_1)-\mathbf{C}_n(\lambda_1)\|_{\sss 2}^2\leq \|\bar{\mathbf{C}}_n(\lambda_1)\|_{\sss 2}^2-\|\mathbf{C}_n(\lambda_1)\|_{\sss 2}^2,
\end{equation}where $\|\cdot\|_{\sss 2}$ denote the $\ell^2$-norm. The final ingredient is the following straightforward lemma:

\begin{lemma}\label{lem_ordering_convergence} Suppose $X_n$, $Y_n$ are non-negative random variables such that $X_n\leq Y_n$ a.s. and $X_n\xrightarrow{\mathcal{L}}X$, $Y_n\xrightarrow{\mathcal{L}}X$. Then, $$Y_n-X_n\xrightarrow{\mathbbm{P}}0.$$
\end{lemma}
\begin{proof}
 Note that  $((X_n,Y_n))_{n\geq 1}$ is tight in $\mathbbm{R}^2$. Thus, for any $(n'_i)_{i\geq 1}$ there exists a subsequence $(n_i)_{i\geq 1}\subset (n'_i)_{i\geq 1}$ such that $(X_{n_i},Y_{n_i})\xrightarrow{\sss \mathcal{L}}(Z_1,Z_2).$  Using the marginal distributional limits we get $Z_1\stackrel{\sss \mathcal{L}}{=} X$, $Z_2\stackrel{\sss \mathcal{L}}{=} X$.  Also the joint distribution of $(Z_1,Z_2)$ is concentrated on the line $y=x$ in the $xy$ plane.  Thus, $(X_{n_i},Y_{n_i})\xrightarrow{\sss \mathcal{L}} (X,X)$. This limiting distribution does not depend on the subsequence $(n_i)_{i\geq 1}$. Thus the tightness of $((X_n,Y_n))_{n\geq 1}$ implies $(X_{n},Y_{n})\xrightarrow{\sss \mathcal{L}} (X,X)$. The proof is now complete.
\end{proof}
\noindent Now, observe that $\|\mathbf{C}_n(\lambda_1)\|_{\sss 2}^2\leq \|\bar{\mathbf{C}}_n(\lambda_1)\|_{\sss 2}^2$ and $\|\mathbf{C}_n(\lambda_1)\|_{\sss 2}^2$, and  $\|\bar{\mathbf{C}}_n(\lambda_1)\|_{\sss 2}^2$ have the same distributional limit by Theorem~\ref{thm_surplus}, and \eqref{joint-conv-2}.  Thus, Lemma~\ref{lem_ordering_convergence} implies that $\|\bar{\mathbf{C}}_n(\lambda_1)\|_{\sss 2}^2-\|\mathbf{C}_n(\lambda_1)\|_{\sss 2}^2\xrightarrow{\sss\mathbbm{P}} 0$, and \eqref{joint-conv-2}, \eqref{bound-norm} yield 
\begin{equation}
 (\mathbf{C}_n(\lambda_0),\mathbf{C}_n(\lambda_1)) \dto \sqrt{\nu}(\tilde{\bld{\gamma}}^{\lambda_0},\tilde{\bld{\gamma}}^{\lambda_1}).
\end{equation}
Finally, the proof of Theorem~\ref{thm_multiple_convergence} is complete by applying Proposition~\ref{prop:coupling-whp}.\qed

\section*{Acknowledgement} This research has been supported by the Netherlands Organisation for Scientific
Research (NWO) through Gravitation Networks grant 024.002.003. In addition, RvdH has been supported by VICI grant 639.033.806 and JvL has been supported by the European Research Council (ERC).

\bibliographystyle{apa}
\bibliography{project1}

\end{document}